\documentclass[12pt]{article}
\usepackage[english]{babel}
\usepackage{subfig}
\usepackage{hyperref,url, enumerate}
\usepackage[numbers,sort&compress]{natbib}
\usepackage{amsfonts}
\usepackage{mathrsfs} 
\usepackage{appendix}
\usepackage{mathtools}
\usepackage{amsmath}
\usepackage{amssymb}
\usepackage{amsthm}
\usepackage{bbm}
\usepackage{latexsym}
\usepackage{natbib}
\usepackage[mathscr]{euscript}
\usepackage{graphicx}
\usepackage{tikz}
\usepackage{pgfplots}
\usepackage{pgfplotstable}
\usetikzlibrary{arrows,positioning,chains,fit,shapes,calc}
\newtheorem{theorem}{Theorem}[section]
\newtheorem{lemma}[theorem]{Lemma}
\newtheorem{corollary}[theorem]{Corollary}
\newtheorem{proposition}[theorem]{Proposition}

\newtheorem{fact}[theorem]{Fact}
\newtheorem{example}[theorem]{Example}
\newtheorem{remark}[theorem]{Remark}

\newtheorem{conjecture}[theorem]{Conjecture}

\newtheorem{definition}[theorem]{Definition}

\pgfplotsset{compat=1.8}

\def\NN{{\mathbb N}}

\DeclareMathOperator{\diam}{diam}

\DeclareMathOperator*{\ex}{ex}

\def\epsilon{\varepsilon}
\let\eps=\varepsilon

\renewcommand{\subset}{\subseteq}
\renewcommand{\dots}{\ldots}

\title{Degree conditions for embedding trees}
\date{\today}

\author{Guido Besomi, Mat\'ias Pavez-Sign\'e\footnote{MPS was supported by CONICYT Doctoral Fellowship 21171132.}, and Maya Stein\footnote{MS is also affiliated to Centro de Modelamiento Matem\'atico, Universidad de Chile, UMI 2807 CNRS. MS acknowledges support by CONICYT + PIA/Apoyo a centros cient\'ificos y tecnol\'ogicos de excelencia con financiamiento Basal, C\'odigo AFB170001 and by Fondecyt Regular Grant 1183080.}\\ \ \\
Departamento de Ingenier\'ia Matem\'atica\\ Universidad de Chile\\  Beauchef 851\\ Santiago, Chile\\
\texttt{\{gbesomi, mpavez, mstein\}@dim.uchile.cl}
}

\begin{document}

\maketitle

\begin{abstract}
We conjecture that every $n$-vertex graph of minimum degree at least $\frac k2$ and maximum degree at least $2k$ contains all trees with $k$ edges as subgraphs. We prove an approximate version of this conjecture for trees of bounded degree and dense host graphs. 

Our result relies on a general embedding tool for embedding trees into graphs of certain structure. This tool also has implications on the Erd\H os--S\'os conjecture and the $\frac 23$-conjecture. We prove an approximate version of both conjectures for  bounded degree trees  and dense host graphs.
\end{abstract}

\newpage
\thispagestyle{empty}
\addtocontents{toc}{\protect\enlargethispage{3.4\baselineskip}}
\tableofcontents
\newpage

\section{Introduction}

A central problem in graph theory consists of determining which conditions a graph $G$ has to satisfy in order to ensure  it contains a given subgraph $H$. For instance, conditions on the average degree, the median degree or the minimum degree of $G$ that ensure that $H$ embeds in $G$ have been studied extensively. Classical examples include Tur\'an's theorem for  containment of a complete subgraph, or Dirac's theorem for containment of a Hamilton cycle. In this paper we will focus on degree conditions that ensure the embedding of all trees of some given size.

Given $k\in\mathbb{N}$, a greedy embedding argument shows that in a graph $G$ with minimum degree at least $k$ one may embed every tree with $k$ edges. Note that the bound on the minimum degree is tight, as one easily confirms considering the example given by the union of several disjoint copies of $K_k$ (the complete graph on $k$ vertices), which does not contain any tree with $k$ edges. Another example is given by any $(k-1)$-regular graph and the star $K_{1,k}$.

A classical conjecture of Erd\H{o}s and S\'os from 1963 suggests that it is possible to replace the minimum degree condition with a bound on the average degree. The bound is again tight by the examples given above. 
	\begin{conjecture}[Erd\H{o}s and S\'os~\cite{Erdos64}]\label{ESconj} Let $k\in\NN$. Every graph with average degree greater than $k-1$ contains every tree with $k$ edges as a subgraph.
		\end{conjecture} 
The Erd\H{o}s-S\'os conjecture is trivially true for stars, and holds for paths by a result of Erd\H{o}s and Gallai~\cite{Erdos1959}.  In the early 1990's Ajtai, Koml\'os, Simonovits and Szemer\'edi announced a solution of the Erd\H{o}s-S\'os Conjecture for large graphs. Nevertheless, the conjecture has received a lot of attention over the last two decades, see for instance~\cite{bradob, hax, sacwoz, goerlich2016}.\\

It is well known that in every graph of average degree greater than $k$ one can find a subgraph of minimum degree greater than $\frac k2$ and average degree greater than~$k$, by successively deleting vertices of too low degree. So, one can assume that the host graph from  the Erd\H os--S\'os conjecture has minimum degree greater than $\frac k2$.

In this direction, Bollob\'as~\cite{bollobas1978} conjectured in 1978 that any  graph on $n$ vertices and minimum degree at least $(1+o(1))\frac{n}{2}$ would contain every spanning tree with maximum degree bounded by a constant. This conjecture was proved by Koml\'os, S\'ark\"ozy and Szemer\'edi~\cite{KSS95} in 1995, giving one of the earliest applications of the Blow-up lemma. 
Csaba, Levitt, Nagy-Gy\"orgy and Szemer\'edi~\cite{CLNS10}  show in 2010 that actually,  a minimum degree of at least $\frac{n}{2}+c\log n$ suffices. 

In 2001, Koml\'os, S\'ark\"ozy and Szemer\'edi~\cite{KSS2001} improved their earlier result   in a different direction, showing that one can actually  embed spanning trees with maximum degree of order $O(\frac{n}{\log n})$. Let us state their result here.

\begin{theorem}[Koml\'os, S\'ark\"ozy and Szemer\'edi~\cite{KSS2001}]\label{thm:KSS} For all $\delta\in(0,1)$, there are $n_0$ and $c\in(0,1)$ such that the following is true for all graphs $G$ and for all trees $T$ with $|V(G)|=|V(T)|=n\ge n_0$.\\ If  $\Delta(T)\le c\frac{n}{\log n}$ and  $\delta(G)\ge (1+\delta)\frac{n}{2}$, then $T$ is a subgraph of $G$.
\end{theorem}

They also show that their bound on the maximum degree is essentially best possible~\cite{KSS2001}.

\smallskip

A natural question is whether a version of Theorem~\ref{thm:KSS} holds for trees that are not necessarily spanning. That is, one would replace $n$ by $k$ in the minimum degree condition for $G$, for some $k<n$,  and hope that $G$ would contain every tree with $k$ edges (or at least each such tree of bounded degree). Clearly, this cannot work, because of the example given before Conjecture~\ref{ESconj}, or in fact, one could consider the union of disjoint copies of~$K_{\ell}$, for any $\ell$ with $(1+\delta )\frac k2+1\le\ell\leq k$. 

However, we believe that if in addition to the minimum degree condition, we require $G$ to have at least one vertex of large degree, then  every tree with $k$ edges should be contained in $G$.   More precisely, we believe that the following holds. 

\begin{conjecture}\label{2k,k/2}Let $k\in\mathbb{N}$ and let $G$ be a graph of minimum degree at least $\tfrac{k}{2}$ and maximum degree at least $2k$. Then every tree with $k$ edges is a subgraph of $G$.
\end{conjecture} 

Conjecture~\ref{2k,k/2} is essentially tight due to the following example. 

\begin{example}\label{example1}
Given $\varepsilon>0$ and $k\in\mathbb N$, let $G_{\varepsilon,k}$ consist of two copies of the complete bipartite graph with parts of size $(1-\varepsilon)k$ and $(1-\varepsilon)\tfrac{k}{2}$, and one vertex that is adjacent to every vertex in the parts of size $(1-\varepsilon)k$. It is easy to see that $G_{\varepsilon,k}$ does not contain the tree $T_k$ consisting of $\sqrt{k}$ stars of size $\sqrt{k}$ whose centers are adjacent to the central vertex of $T_k$, provided that $k=k(\varepsilon)$ is sufficiently large.
\end{example}

\begin{figure}[h!]
	\centering	
	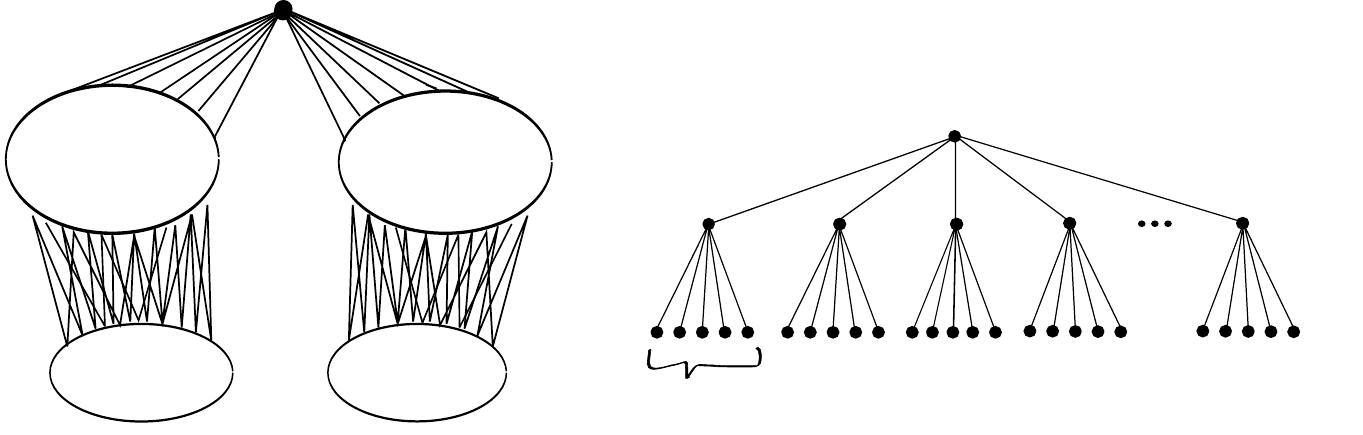\\
	\caption{Graph $G_{\varepsilon, k}$ and tree $T_k$ from Example~\ref{example1}.}
	\label{fig:extremal}
	\begin{picture}(0,0)
\put(-187,110){\colorbox{white}{$(1-\varepsilon)k$}}
\put(-91,110){\colorbox{white}{$(1-\varepsilon)k$}}
\put(-175,50){$(1-\varepsilon)\frac k2$}
\put(-96,50){$(1-\varepsilon)\frac k2$}
\put(-8,37){$\sqrt{k}-1$}
\end{picture}
\end{figure}
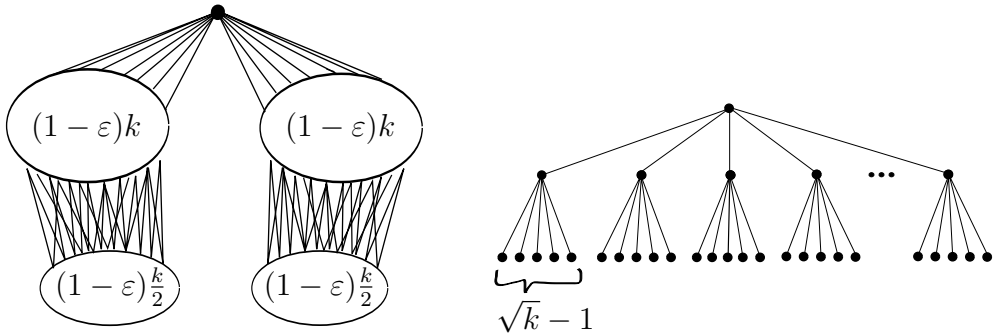

An even stronger argument for the tightness of Conjecture~\ref{2k,k/2} is given in~\cite{BPS3}\footnote{There we prove that  for all $\eps>0$ there are $k\in \mathbb N$, a $k$-edge tree~$T$, and a graph $G$ with $\delta(G)\geq  \frac k2$ and $\Delta (G)\geq 2(1-\eps)k$ with the property that $T\not\subseteq G$. The underlying example is very similar to Example~\ref{example1}, but we take a little more care when choosing the number and size  of the stars that make up the bulk of the tree $T_k$.}.

\medskip

 Observe that Conjecture~\ref{2k,k/2} trivially holds for both the star and the double star. Also, it is easy to see that it holds for paths: If the host graph $G$ has a $2$-connected component of size at least $k+1$, then by a well-known theorem of Dirac, this component contains a cycle of length at least $k$, and thus also a $k$-edge path (possibly using one edge that leaves the cycle). Otherwise, we can embed a vertex from the middle of the path into any cutvertex $x$ of $G$, and then greedily embed the remainder of the path into two components of $G-x$, using the minimum degree of $G$.

\smallskip

As more evidence for Conjecture~\ref{2k,k/2}, we prove an approximate version for trees of bounded degree and dense host graphs. 
\begin{theorem}\label{main:1} 
For every $\delta >0$ there is $n_0\in\NN$ such that for all $n\ge n_0$ the following holds for  every $k$ with $n\ge k\ge \delta n$.\\
If $G$ is an $n$-vertex graph of minimum degree at least $(1+\delta)\frac{k}{2}$ and maximum degree at least $2(1+\delta)k$, then $G$ contains every $k$-edge tree $T$ of maximum degree at most $k^{\frac{1}{67}}$ as a subgraph. 
\end{theorem}
 
Moreover, if  we consider trees whose maximum degree is bounded by an absolute constant, 
 we can improve the bound on the maximum degree of the host graph given by Theorem~\ref{main:1} as follows. 
 
\begin{theorem}\label{main:1-2}
For every $\delta >0$  and $\Delta\geq2$ there is $n_0\in\NN$ such that for all $n\ge n_0$ the following holds for every $k$ with $n\ge k\ge \delta n$.\\
 If $G$ is an $n$-vertex graph  of minimum degree at least $(1+\delta)\frac{k}{2}$ and maximum degree at least $2(\tfrac{\Delta-1}{\Delta}+\delta)k$, then $G$ contains every tree $T$ with $k$ edges and maximum degree at most $\Delta$ as a subgraph. 
\end{theorem}
		
		The bounds on the maximum and minimum degree of $G$ are close to best possible, which can be seen by considering  Example~\ref{exampleDelta} in Section~\ref{sec:const}. 
		We also believe a version of Conjecture~\ref{2k,k/2}, with the maximum degree bound adjusted as in Theorem~\ref{main:1-2}, should hold for constant degree trees  (see Conjecture~\ref{conj:const} in Section~\ref{sec:const}).

\smallskip

This is not the first time a combination of a minimum and maximum degree condition has been proposed for replacing the average degree condition in the Erd\H os--S\'os conjecture. In 2016,  Havet, Reed, Stein, and Wood put forward the following conjecture.

\begin{conjecture}[$\frac 23$-conjecture; Havet, Reed, Stein and Wood~\cite{2k3:2016}]\label{23conj}
	Let $k\in\mathbb{N}$ and let $G$ be a graph of minimum degree at least $\lfloor{\frac{2k}{3}}\rfloor$ and maximum degree at least~$k$. Then every tree with $k$ edges is a subgraph of $G$. 
\end{conjecture}

In~\cite{2k3:2016},  two variants of this conjecture are shown to be true: First, if the bound on the maximum degree  is replaced with a (large) function of~$k$; second, if the bound on the minimum degree  is replaced with  $(1-\gamma)k$, for a very small but explicit~$\gamma>0$. Moreover, Reed and Stein show in~\cite{RS18a, RS18b} that Conjecture~\ref{23conj} holds for large $k$, in the case of spanning trees (that is, if we additionally assume that $|V(G)|=|V(T)|=k+1$).

\smallskip

The tools developed for the proof of Theorem~\ref{main:1}, namely the key embedding result Lemma~\ref{lem:superlemma}, allow us to show  an approximate version of the $\frac{2}{3}$-conjecture for dense host graphs and trees with bounded maximum degree. 

\begin{theorem}\label{main:2} 
For every $\delta >0$ there is $n_0\in\NN$ such that for each $k$ and for each $n$-vertex graph $G$  with $n\ge n_0$ and $n\ge k\ge \delta n$  the following holds. If $G$ has minimum degree at least $(1+\delta)\frac{2k}{3}$ and maximum degree at least $(1+\delta)k$, then $G$ contains every $k$-edge tree of maximum degree  at most $k^{\frac{1}{49}}$. 
\end{theorem}

Even before reaching the main steps in our proof of Theorem~\ref{main:1}, we can use the first stepping stones of this proof to
deduce an approximate version of the Erd\H{o}s-S\'os conjecture, for  bounded degree trees and dense host graphs. 

\begin{theorem}\label{thm:ESap}
	For every $\delta>0$ there is $n_0\in\NN$ such that for each $k$ and for each $n$-vertex graph $G$  with $n\ge n_0$ and $n\ge k\ge \delta n$  the following holds. If  $d(G)>(1+\delta) k$, then $G$ contains every $k$-edge tree  of maximum degree  at most $k^{\frac{1}{67}}$ as a subgraph.	
\end{theorem}

An improvement of this result has been obtained by Rozho\v{n}~\cite{rohzon}, and independently in~\cite{thesisGuido, BPS2}.

An outline of the proofs of  all our results will be given in  Section~\ref{outline}, while the actual proofs will be postponed to Section~\ref{improve} (Theorem~\ref{thm:ESap}) and to Section~\ref{s:mindegree} (Theorems~\ref{main:1}, ~\ref{main:1-2} and~\ref{main:2}). The proofs all rely on the regularity approach to embedding trees, and all but Theorem~\ref{thm:ESap} rely on a significant amount of additional work. This includes results on cutting and arranging trees, as well as a structural embedding result  for bounded degree trees, namely Lemma~\ref{lem:superlemma}. 

Our key embedding lemma,  Lemma~\ref{lem:superlemma}, is stated in a very general way, and covers a variety of possible host graphs. For this reason, we believe that it will be useful for future work on tree embeddings with minimum degree conditions.

	\smallskip
	
	In the Conclusion of this paper (Section~\ref{sec:conj}),  we provide some further results, conjectures and examples in the main direction of this paper. 
	
	In Section~\ref{ell} we generalise Conjecture~\ref{2k,k/2}, contemplating the whole range of values in $[k, 2k]$ for the maximum degree and in $[\frac k2, k]$ for the minimum degree of the host graph $G$. 
		 In Sections~\ref{4/3} and~\ref{2nd} we discuss extensions of Theorem~\ref{main:1} and of Conjecture~\ref{2k,k/2} to graphs that do not quite satisfy the degree conditions of the theorem, but instead fulfill a structural condition. Section~\ref{sec:const} 
		 discusses trees of constant maximum degree.
		  
		 In Sections~\ref{supersat} and~\ref{resil}, we discuss supersaturation for our theorems, and determine the extremal graphs for our results. These extremal graphs turn out to be the known examples for the Erd\H os--S\'os conjecture for Theorem~\ref{thm:ESap}, the known examples  for the $\frac 23$-conjecture plus 
Example~\ref{example1} 
 for Theorem~\ref{main:2}, and graphs as in  Example~\ref{example1} for Theorems~\ref{main:1} and~\ref{main:1-2}.

	\section{Outline of the proof}\label{outline}
	
	The aim of this section is  to give the general idea of the structure of the proof of our results. 
	For the understanding of the paper, it is not necessary to read the present section (but we hope it will be helpful).
	
	\medskip
	
	Before we explain the ideas let us give a very short overview of the structure of the paper.
Our main result, Theorem~\ref{main:1}, as well as Theorems~\ref{main:1-2} and~\ref{main:2} will be shown in Section~\ref{s:mindegree}. Their proofs rely on a structural embedding result, namely Lemma~\ref{lem:superlemma}, which is shown in Section~\ref{sec:general_lemma}. This lemma in turn relies on  results on tree-cutting from Section~\ref{sec:trees} and on tree embedding results from Sections~\ref{sec:proof1} and~\ref{improve}. The results from Sections~\ref{sec:proof1} and~\ref{improve} on their own already imply Theorem~\ref{thm:ESap}. In Section~\ref{sec:prelim}, we discuss some preliminaries (regularity and a matching result), which will be needed for Section~\ref{sec:proof1}.

	\bigskip
	
	Let us now expose the  general idea of the proof. As most of our results rely on our key embedding lemma, Lemma~\ref{lem:superlemma}, let us start by describing this lemma, which will be 
	 stated and proved 
in Section~\ref{sec:general_lemma}. 
		
		Lemma~\ref{lem:superlemma} provides an embedding of any tree~$T$ with maximum degree bounded by~$k^\frac{1}{c}$, where $c$ is a constant,\footnote{This is a different constant in the situation where the minimum degree is bounded by  $(1+\delta)\frac k2$, as in Theorem~\ref{main:1},  than in the situation where we have minimum degree at least $(1+\delta)\frac 23k$, as in Theorem~\ref{main:2}.}  into any host graph~$G$ of suitable minimum degree, as long as $G$ contains one of several favourable scenarios explicitly described in the statement of Lemma~\ref{lem:superlemma}. 
		
		The scenarios contemplated by  the lemma cover the situation where, after applying the regularity lemma\footnote{For an introduction to regularity, see Section~\ref{sec:regu}.}  to  $G$, the corresponding reduced graph has a large\footnote{That is, large enough to accommodate $T$.} component, but also cover a number of situations where there is no large component. In these latter situations, we will have to use a maximum degree vertex $x$ of~$G$, as well as a suitable cutvertex~$z$ of~$T$, and embed the components of $T-z$ into components of $G-x$. Several possible shapes and sizes of components possibly seen by $x$ are taken into account in Lemma~\ref{lem:superlemma}. We believe that this generality could make Lemma~\ref{lem:superlemma} very useful for future work on tree embeddings with minimum degree conditions.
		
Once we have Lemma~\ref{lem:superlemma}, the proof of  Theorems~\ref{main:1},~\ref{main:1-2} and~\ref{main:2} will be fairly easy.  We only need to regularise the host graph $G$, and show we are in one of the situations as described in Lemma~\ref{lem:superlemma}. This is done in Section~\ref{s:mindegree}. (Theorem~\ref{thm:ESap}, too, could be deduced from the lemma, but its proof already follows from the preparatory steps earlier in the paper.) 

So let us now sketch the proof of Lemma~\ref{lem:superlemma}.
	There are two crucial ingredients for the proof of Lemma~\ref{lem:superlemma}. One of these ingredients is some work that we accomplish  in Section~\ref{sec:cutT}. In that section, we prove some useful results on cutting trees, the most important ones being  Lemma~\ref{lem:num} and Proposition~\ref{prop:cut5}. These two auxiliary results allow us to cut a tree at some convenient cutvertex~$z$ and then group the components of $T-z$ into two or three groups (as necessary), so that the union of the components from these groups form sets of convenient sizes. Moreover, we show that it is possible to 2-colour $T-z$ in way that the resulting colour classes are not too unbalanced. This will be very important when, in the context of Lemma~\ref{lem:superlemma}, we wish to embed several components of $T-z$ into a single bipartite component of the reduced graph of $G-x$.
	
	The other crucial ingredient for the proof of Lemma~\ref{lem:superlemma} is the preparatory work accomplished in Sections~\ref{sec:proof1} and~\ref{improve}. There we show how to embed a tree into a host graph, that, 
	after an application of the regularity lemma yields a reduced graph with a large connected component. 
	For this, we cut the tree into tiny subtrees and few connecting vertices (a now standard procedure that is explained in the short Section~\ref{tree-cutting}), and then embed these trees into suitable edges of the reduced graph. 
	 The only remaining problem is how to make the connections between the tiny trees. 
	
	For these connections, we use paths in the reduced graph. For this argument to work, we have to bound the maximum degree of the tree we wish to embed in terms of  the diameter of the reduced graph of $G$. (Another argument will allow us to relax the bound later, see below.)
	Also, we have to distinguish two cases, namely whether the large component of the reduced graph is bipartite or not. The reason for this is that we need to fill the edges we embed our tiny tree into in a balanced way. The two cases will be treated in  Propositions~\ref{prop:bipartite} and~\ref{prop:non bipartite}, respectively. In the remainder of Section~\ref{sec:proof1}, we  deduce some corollaries from these propositions, which will come in handy later, when in the proof of 
	Lemma~\ref{lem:superlemma},
	 we need to embed parts of the tree into parts of the host graphs that correspond to different components of its reduced graph. 
	
	In Section~\ref{improve}, we unify and improve the results from Section~\ref{sec:proof1}. Namely, in Proposition~\ref{prop:connected-cte} we provide an embedding result for trees into large connected components of the reduced graph of $G$, where the bound on the maximum degree of the tree no longer depends on the diameter of the reduced graph of the host graph, but instead is $k^\frac{1}{c}$, where $c$ is an absolute constant. The idea for the proof of this result is that we first try to follow the embedding scheme from the previous section, but only using paths of bounded length for the connections. Should this fail, then the only possible reason is that we could not  reach suitable free space at a bounded distance from the cluster~$C$ we were currently embedding into. In this case, we abort our mission, and are able to prove that it is possible to embed the tree into a ball of appropriate radius centered at $C$.
	
	Also in Section~\ref{improve}, we show  our approximate version of the Erd\H os-S\'os conjecture for bounded degree trees,  Theorem~\ref{thm:ESap}. It  easily follows from the results obtained in Sections~\ref{sec:proof1} and~\ref{improve}.

	\section{Preliminaries}\label{sec:prelim}
\subsection{Notation}
Given $\ell\in\mathbb{N}$, we write $[\ell]=\{1,\dots,\ell\}$. Also, we will write $a\ll b$ to indicate that  given $b$, we choose $a$ significantly smaller. The explicit value for such $a$ can be calculated from the proofs.

We write $|H|=|V(H)|$ for the number of vertices  and $e(H)=|E(H)|$ for the number of edges of a graph $H$, and $\delta(H)$, $d(H)$ and $\Delta(H)$ are the minimum, average and maximum degree of $H$.  As usual, $\deg_H(x)$ is the degree of vertex $x\in V(H)$, and we write $N_H(x)$ for its neighbourhood in $H$,   $N_H(x,S)=N_H(x)\cap S$ for its neighbourhood in $S\subseteq V(H)$ and $\deg_H(x,S)$ for the respective degree. For disjoint sets $X,Y\subseteq V(H)$, we write $E_H(X,Y)$ for the set of edges $xy\in E(H)$ with $x\in X$ and $y\in Y$ and set $e_H(X,Y):=|E_H(X,Y)|$. 
In all of the above, we omit the subscript $H$ if it is clear from the context. 
 Given $U\subset V(H)$ we write $H[U]$ for the graph induced in~$H$ by the vertices in $U$, and we say a vertex $x$ {\em sees} $U$ if it sends an edge to $U$.

\subsection{The regularity lemma}\label{sec:regu}
In the present section, we discuss the notion of regularity and a few basic facts. Readers familiar with this topic are invited to skip the section.

Let $H=(A,B)$ be a bipartite graph with density $d(A,B):=\frac{e(A,B)}{|A||B|}$. For a fixed $\varepsilon>0$, the pair $(A,B)$ is said to be {\em $\varepsilon$-regular} if for any $X\subseteq A$ and $Y\subseteq B$, with $|X|>\varepsilon|A|$ and $|Y|> \varepsilon |B|$, we have that 
$$|d(X,Y)-d(A,B)|<\varepsilon.$$
Moreover, an $\varepsilon$-regular pair $(A,B)$ is called $(\varepsilon,\eta)$-regular if $d(A,B)>\eta$. Given an $\varepsilon$-regular pair $(A,B)$, we say that  $X\subseteq A$ is \textit{$\varepsilon$-significant} if $|X|> \varepsilon |A|$, and similar for subsets of $B$. A vertex $x\in A$ is called \textit{$\varepsilon$-typical} to a significant set $Y\subseteq B$ if $\deg(x,Y)> (d(A,B)-\varepsilon)|Y|$. We simply write \textit{regular}, \textit{significant} or \textit{typical} if $\varepsilon$ is clear from the context.

It it well known that regular pairs behave, in many ways, like  random bipartite graphs with the same edge density. 
The next well known fact (see for instance~\cite{regu}) states that in a regular pair almost every vertex is typical to any given significant set, and also that regularity is inherited by subpairs. 

\begin{fact}\label{fact:1}Let $(A,B)$ be an $\varepsilon$-regular pair with density $\eta$. Then the following holds:
\begin{enumerate}
 	\item For any $\varepsilon$-significant $Y\subseteq B$, all but at most $\varepsilon|A|$ vertices from $A$ are $\varepsilon$-typical to $Y$.
 	\item\label{fact:1,2} Let $\delta\in (0,1)$. For any subsets $X\subseteq A$ and $Y\subseteq B$, with $|X|\ge\delta|A|$ and $|Y|\ge\delta|B|$, the pair $(X,Y)$ is $\frac{2\varepsilon}{\delta}$-regular with density between $\eta-\varepsilon$ and $\eta+\varepsilon$.
\end{enumerate}

\end{fact}

The regularity lemma of Szemer\'edi~\cite{Sze78} states that, for any given $\varepsilon>0$, the vertex set of any large enough graph can be partitioned into a bounded number of sets, also called clusters, such that the graph induced by almost any pair of these clusters is $\varepsilon$-regular. 
We will make use of the well known degree form of the regularity lemma  (see for instance~\cite{regu}).
Call a vertex partition $V(G)=V_1\cup\dots\cup V_\ell$ an {\em $(\varepsilon,\eta)$-regular partition} if 
\begin{enumerate}
	\item $|V_1|=|V_2|=\dots=|V_\ell|$;
	\item  $V_i$ is independent for all $i\in[\ell]$; and
	\item for all $1\le i<j\le \ell$, the pair $(V_i,V_j)$ is $\varepsilon$-regular with density either $d(V_i,V_j)>\eta$ or $d(V_i,V_j)=0$.
\end{enumerate}

\begin{lemma}[Szemer\'edi's regularity lemma - Degree form]\label{reg:deg}
For all $\varepsilon>0$ and $m_0\in\NN$ there are $N_0, M_0$ such that the following holds for all $\eta\in[0,1]$ and $n\ge N_0$. Any $n$-vertex graph $G$ has a subgraph $G'$, with $|G|- |G'|\le \varepsilon n$ and $\deg_{G'}(x)\ge \deg_G(x)-(\eta+\varepsilon)n$ for all $x\in V(G')$, such that $G'$ admits an $(\varepsilon,\eta)$-regular partition $V(G')=V_1\cup\dots\cup V_\ell$, with $m_0\le \ell\le M_0$.
\end{lemma}

The {\em $(\varepsilon,\eta)$-reduced graph} $R$, with respect to the $(\varepsilon,\eta)$-regular partition given by Lemma~\ref{reg:deg}, is the graph with vertex set $\{V_i:i\in[\ell]\}$, called \textit{clusters}, in which $V_iV_j$ is an edge if and only if $d(V_i,V_j)>\eta$. We will sometimes refer to the $(\varepsilon,\eta)$-reduced graph $R$ without explicitly referring to the associated $(\varepsilon,\eta)$-regular partition of the graph.

 It turns out that $R$ inherits many  properties of $G'$, such as the edge density or the minimum degree (scaled to the order of $R$). Indeed, it is easy to calculate the following.
\begin{fact}\label{fact:2}Let $0<2\varepsilon\le\eta\le \tfrac{\alpha}{2}$. If $G$ is a $n$-vertex graph with $\delta(G)\ge \alpha n$, and $R$ is an $(\varepsilon,\eta)$-reduced graph of $G$, then $\delta(R)\ge (\alpha-2\eta)|R|$. 
	\end{fact}
	
\smallskip

We close this subsection with a lemma that illustrates why regularity is useful for embedding trees. It states that a tree will always fit into a regular pair, if the  tree is small enough.
	
\begin{lemma}\label{lem:T1}
Let $0<\beta\le\varepsilon\le \tfrac{1}{25}$. Let $(A,B)$ be a $(\varepsilon,5\sqrt{\varepsilon})$-regular pair with $|A|=|B|=m$, and let $X\subseteq A, Y\subseteq B, Z\subseteq A\cup B$ be such that $\min\{|X\setminus Z|, |Y\setminus Z|\}>\sqrt \varepsilon m$. \\
Then any tree $T$ on at most $\beta m$ vertices can be embedded into $(X\cup Y)\setminus Z$.
Moreover, for each $v\in V(T)$ there are at least $2\varepsilon m$ vertices from $(X\cup Y)\setminus Z$ that can be chosen as the image of $v$.
\end{lemma}

\begin{proof}
We construct the embedding $\phi:V(T)\to X\cup Y$ levelwise, starting with the root, which is embedded into a typical vertex of $(X\cup Y)\setminus Z$. At each step~$i$ we ensure that all vertices of level $i$ are embedded into vertices of $X\setminus Z$ (or~$Y\setminus Z$) that are typical with respect to the unoccupied vertices of $Y\setminus Z$ (or $X\setminus Z$). This is possible, because at each step $i$, and for each vertex $v$ of level $i$, the degree of a typical vertex into the unoccupied vertices on the other side is at least $4\varepsilon m$, 
and there are at most $\varepsilon m$ non typical vertices and at most $|T|\leq \beta m$ already occupied vertices.
\end{proof}

%
\subsection{A matching lemma}
We finish the preliminaries  with a lemma that will allow us to find a large connected matching in the reduced graph. 
\begin{lemma}\label{lem:matching}
Let $H$ be any graph. Then there exists an independent set~$I$, a matching $M$, and a set of vertex disjoint triangles $\Gamma$ so that $V(H)=I\cup V(M)\cup V(\Gamma)$. Moreover,  there is a partition $V(M)=V_1\cup V_2$ of $V(M)$ such that every edge of $M$ has one vertex in $V_1$ and one vertex in $V_2$, and $N(x)\subseteq M_1$ for all $x\in I$.
\end{lemma}

\begin{proof}
Let us choose  a  matching $M$ and a family $\Gamma$ of disjoint triangles, that are disjoint from $M$, maximising $|V(M)|+|V(\Gamma)|$. Then the set~$I$ consisting of all vertices not covered by $M\cup\Gamma$ is independent. 

Consider a vertex $x\in I$. Note that because of our choice of $M$ and $\Gamma$, we know that $x$ is not adjacent to any vertex from any  triangle from $\Gamma$. Also, note that for any edge $uv$ in $M$, vertex $x$ sees at most one of $u$, $v$. Finally, if $x$ sees $u$, then no other vertex from $I$ can see $v$. This proves the statement.
	\end{proof}

\section{Cutting trees}\label{sec:trees}

This section contains some preliminary results on cutting trees.  In the first subsection we show how to find a constant number of vertices, such that after taking these out, the remaining components are tiny. In fact, they will be so small that we can use Lemma~\ref{lem:T1} in order to embed them into a regular pair.

  In the second subsection, we find a cutvertex $z$ and a colouring of the remaining components, so that the colour classes of each component are rather balanced (we are able to guarantee a ratio of roughly $\frac 13$--$\frac 23$ at worst). This is very useful when we are in the situation where the reduced graph has no big component. Then, we can embed the cutvertex $z$ into a maximum degree vertex of $G$, and embed the components of $T-z$ into the components of the reduced graph. Since we balanced them, and because of the minimum degree of $G$, they will fit and the embedding can be completed.


\smallskip

Let us go through some notation for trees. Given any rooted tree $T$, we define a partial order $\preceq$ on $V(T)$ by saying that $x\preceq y$ ($x$ is {\it below} $y$) if and only if $x$ lies on the unique path from $y$ to $r(T)$, where $r(T)$ denotes the root of $T$. If in addition, $xy\in E(T)$,  we say $y$ is a {\it child} of $x$, and $x$ is the {\it parent} of $y$. The {\it tree $T(x)$ induced by $x$} is the subtree of $T$ induced by the set $V(T(x))=\{v:v\preceq x\}$. For $i\ge 0$, the $i^{\text{th}}$ level of $T$, denoted by $L_i$, consists of all vertices at distance $i$ from $r(T)$.

\subsection{Cutting $T$ into small trees}\label{tree-cutting}

As we showed  in Section~\ref{sec:regu}, it is simple to embed sufficiently small trees into regular pairs, and furthermore, one may continue embedding small trees until the pair become almost full. For this reason, it is useful to cut down the tree $T$ one would like to embed into a regularised host graph $G$ into very small subtrees, connected by few vertices.

The next proposition shows that one can decompose any tree into a family of small subtrees. Versions of this proposition have already appeared in earlier literature on tree embeddings, for instance in~\cite{AKS}. 

\begin{proposition}\label{prop:cut3}
	Let $\beta\in(0,1)$, and let $T$ be a rooted tree on $t+1$ vertices. Then there exists a set  $S\subset V(T)$ and a family~$\mathcal{P}$ of disjoint rooted trees such that 
	\begin{enumerate}[(i)]
	\item $r(T)\in S$;
		\item $\mathcal{P}$ consists of the components of $T-S$, and each $P\in\mathcal P$ is rooted at the vertex closest to the root of $T$;
		\item $|P|\le\beta t$ for each  $P\in\mathcal{P}$; and
		\item $|S| < \frac{1}{\beta}+2$.
	\end{enumerate}
	The vertices from $S$ will be called {\em seeds}, and the components from $\mathcal P$ will be called the {\em pieces} of the decomposition.
\end{proposition}

\begin{proof}
 We iteratively construct the set $S$, starting with $T^0:=T$ and $S^0:= \emptyset$. In step $i+1$, let $s_{i+1}$ be the maximal vertex of $T^i$ such that
	\begin{equation*}\label{eq:iter}
		|T^{i}(s_{i+1})| > \beta t.
	\end{equation*}
Note that by the maximality of $s_{i+1}$ the trees in $T^{i}(s_{i+1})-s_{i+1}$ each cover at most $\beta t$ vertices. Obtain $S_{i+1}$ by adding $s_{i+1}$ to $S^{i}$ and set $T^{i+1} = T^i-T^i(s_{i+1})$. 
If at some step $j$ there is no vertex $s_{j+1}$ with $|T^{j}(s_{j+1})| > \beta t$, then $|T^j|\leq \beta t$, and we  end the process. We set $S:=S^{j}\cup\{r(T)\}$ and let $\mathcal{P}$ be the set of connected components of $T- S$.
	
	Properties (i)--(iii) clearly hold. For (iv) observe that $|T^{i+1}|< |T^i|-\beta t$. Hence,  $$0 \leq |T^m| < |T^0| - j\cdot\beta t ,$$ which in turn implies that $|S| = j+1 \leq \frac{|T|}{\beta t}+1< \frac{1}{\beta}+2$.   
\end{proof}

\subsection{Finding a good cutvertex}\label{sec:cutT}

The objective of this subsection is to show several useful lemmas on cutting trees at one vertex. The first and easiest two lemmas allow us to find a cutvertex $z$ in a given tree $T$ in a way that we have some control on the size of the components of $T-z$ (Lemma~\ref{lem:cut1}), and moreover, allow us to group these components having some control over the total sizes of the groups (Lemma~\ref{lem:num}). The most laborious result of this section,  Proposition~\ref{prop:cut5}, states that in any tree $T$, we can find a cutvertex and properly $2$-colour the remaining components so that in total the colour classes remain somewhat balanced.

Our first lemma is folklore, but for completeness, we include its proof.

\begin{lemma}\label{lem:cut1}
	Let $T$  be a tree on $t+1$ vertices, and let $x$ be a leaf of $T$. Then~$T$ has a vertex $z$ such that every component of $T-z$ has at most $\lfloor\frac t2\rfloor$ vertices, except the one containing $x$, which has at most $\lceil\frac t2\rceil$ vertices.
\end{lemma}

\begin{proof}
Let $z$ be a maximal vertex (with respect to the order induced by $x$) such that $|T(z)| > \lfloor\frac t2\rfloor$. Then every component of $T-z$ has at most $\lfloor\frac t2\rfloor$ vertices: This is obvious (from the definition of~$z$) for the components not containing $x$, while the component that contains~$x$ only has $|T|-|T(z)|\leq t+1-(\lfloor\frac t2\rfloor +1)=\lceil\frac t2\rceil$ vertices. 
%
\end{proof}

\begin{definition}
Call a vertex $z$ as in  Lemma~\ref{lem:cut1} a {$\frac t2$-separator.}
\end{definition}

We now prove an auxiliary lemma on partitioning sequences of integers. This lemma will be used in the proofs of both Lemma~\ref{lem:cut4} and Proposition~\ref{prop:cut5}, and also in the proofs of Theorems~\ref{main:1} and~\ref{main:1-2}.

\begin{lemma}\label{lem:num}
Let $m,t \in \NN_+$ and let $(a_i)_{i=1}^m$ be a sequence of positive integers  such that $0 < a_i \leq \lceil\frac t2\rceil$, for each $i\in [m]$, and $\sum_{i=1}^m a_i \leq t$. Then 
\begin{enumerate}[(i)]
	\item there is a partition $\{I_1,I_2,I_3\}$ of $[m]$ such that $\sum_{i\in I_3} a_i \leq \sum_{i\in I_2} a_i \leq \sum_{i\in I_1} a_i\leq \lceil\frac t2\rceil$ 
	; and
	\item there is a partition $\{J_1,J_2\}$ of $[m]$ such that  $\sum_{i\in J_2} a_i \leq \sum_{i\in J_1} a_i \leq \frac{2}{3}t$.
\end{enumerate}
\end{lemma}

\begin{proof}
We first pick a set $I_1\subset[m]$ with $\sum_{i\in I_1} a_i \leq \lceil\frac t2\rceil$ that maximizes the sum. From $[m] \setminus{I_1}$ we extract a second set $I_2$ with $\sum_{i\in I_2} a_i \leq \lceil\frac t2\rceil$ that maximizes the sum. The choice of $I_1$ and $I_2$ ensures that for $I_3:=[m] \setminus{(I_1\cup I_2)}$ it also holds that $\sum_{i\in I_3} a_i \leq \lceil\frac t2\rceil$, and that $\sum_{i\in I_3} a_i \leq \sum_{i\in I_2} a_i$.
Therefore, the sets $I_1, I_2, I_3$ fulfill the conditions in~(i). (Notice that $I_3$, and possibly also $I_2$, may be empty.)  


For~(ii) we proceed as follows. If $I_3=\emptyset$ we just set $J_1 := I_1$ and $J_2:= I_2$, which clearly satisfies~(ii). If $I_3\neq \emptyset$ we define $J_1$ as one of the sets $I_2\cup I_3$ and $I_1$, and $J_2$ as the other set, in a way that $\sum_{i\in J_2} a_i \leq \sum_{i\in J_1} a_i$. Observe that the second part of~(i) implies that  $\sum_{i\in I_2\cup I_3} a_i \leq  \frac{2}{3}t$.
 So again,~(ii) is satisfied.
\end{proof}

\begin{remark}\label{rem:F_3}
Observe that the set $I_3$ from Lemma \ref{lem:num}~(i) has at most one element, because otherwise, due to the maximality of $I_1$ and $I_2$, there would exist $j,k\in I_3$ such that $a_j+\sum_{i\in I_1} a_i>\lceil\frac t2\rceil$ and $a_k+\sum_{i\in I_2} a_i>\lceil\frac t2\rceil$, a contradiction to the fact that $\sum_{i=1}^m a_i \leq t$.
\end{remark}

Lemma \ref{lem:num} tells us that after using Lemma~\ref{lem:cut1} to cut a tree $T$ at a vertex~$z$, we can group the components of $T-z$ in a way that the total size of each group is conveniently bounded. We would now like to say something about the balancedness of the resulting forest, and for this we resort to the concept of vertex colouring.

For a proper $2$-colouring $c:V(G)\rightarrow\{0,1\}$ of a graph $G$ with colours $0$ and $1$,  we define $c_0:=\{v\in V(G)\,:\,c(v)=0\}$ and $c_1:=\{v\in V(G)\,:\,c(v)=1\}$. For better readability, throughout all proofs, we will stick to the convention that $|c_0|\geq |c_1|$ (but this will be restated in each proof). 

\begin{lemma}\label{lem:cut4}
Every tree $T$ on $t+1$ vertices has a vertex $z$ such that $T-z$ admits a proper $2$-colouring $c:V(T-z)\rightarrow\{0,1\}$ with $|c_0|\leq \frac{3t-1}{4}$ and $|c_1|\leq \frac{t}{2}$.
\end{lemma}

\begin{proof}
We apply Lemma \ref{lem:cut1} to obtain a cutvertex $z$ and a forest $T-z$ with components $\{T_i\}_{i=1}^m$ such that $|T_i|\leq \lceil\frac t2\rceil$, for every $i$. We will now use Lemma~\ref{lem:num} in order to group the components of $T-z$. Setting $a_i:=|T_i|$, the lemma yields three sets $I_1$, $I_2$ and $I_3$ such that the forests $F_j := \bigcup_{i\in I_j}T_i$, with $j=1,2,3$, cover at most $\lceil\frac t2\rceil$ vertices each. Also, the forest $F_1$ covers at least $\frac {t}3$ vertices.

For $j=1,2,3$ consider any proper $2$-colouring $c^j$ of the forest $F_j$, with colour classes $c_0^j$ and $c_1^j$, such that $F_1$ and $F_2$ each meet both colours. (This is possible unless $|F_1|$ and/or $|F_2|$ is $1$, and in that case we are done anyway.) For each $j$, we assume that $|c_0^j|\geq |c_1^j|$.\\
	
	We split the remainder of the proof into two cases.\\
	
\noindent\textbf{Case 1: }  $|c_0^1| \geq \frac{3|F_1| -1}{4}$.\\

 In this case, define the colouring $c$ by setting  $c_0 := c_0^1 \cup c_1^2 \cup c_1^3$ and $c_1:=V(T-z)\setminus{c_0} = c_1^1 \cup c_0^2 \cup c_0^3$. Then,
\begin{align*}
|c_0|  = |c_0^1| + |c_1^2| + |c_1^3|  & \leq |F_1| -1 + \frac{|F_2|}{2} + \frac{|F_3|}{2} \\ &= \frac{|F_1|}{2}-1 + \frac{|T-z|}{2} \\ & \leq \frac{t+1}{4}-1 + \frac{t}{2}\\ & \leq \frac{3t-1}{4}.
\end{align*}
Moreover,
$$|c_1|\leq t-|c_0^1|-|c_1^2| \leq t-\frac{3|F_1|-1}{4}-1\leq t-\frac{t-1}{4}-1\leq \frac{3t-1}{4}$$
	
\noindent where the  penultimate inequality comes from the fact that $|F_1|\geq \frac{t}{3}$. Hence, $\max\{|c_0|,|c_1|\}\leq \frac{3t-1}{4}$; renaming the colour classes if necessary we get the desired result.\\
	
\noindent\textbf{Case 2: }  $|c_0^1|< \frac{3|F_1|-1}{4}$. \\

In this case, define the colouring $c$ by setting  $c_0 := c_0^1\cup c_1^2 \cup c_0^3$ and $c_1:=V(T-z)\setminus{c_0} = c_1^1\cup c_0^2 \cup c_1^3$. Then, 
\begin{align*}
|c_0| & < \frac{3|F_1|-1}{4} + \frac{|F_2|}{2} + |F_3| \\
& = \frac{|T-z|}{2} + \frac{|F_1|+2|F_3|-1}{4}\\
& \leq \frac{t}{2} + \frac{|F_1|+|F_2|+|F_3|-1}{4}\\
& = \frac{3t-1}{4},
\end{align*}
and
$$|c_1| \leq \frac{|F_1|}{2}+|F_2|-1+\frac{|F_3|}{2}=\frac{t}{2}+\frac{|F_2|}{2} -1\leq \frac{3t-1}{4}.$$
	
Again we obtain $\max\{|c_0|,|c_1|\}\leq \frac{3t-1}{4}$, and swap colour classes if necessary.
\end{proof}

Let us  remark that the bound~$\frac{3t-1}{4}$ given by Lemma \ref{lem:cut4} for the imbalance of the colouring is best possible, if we insist that the cutvertex is as given by Lemma~\ref{lem:cut1}. This is illustrated by the following example.

For  $t$ divisible by four, consider the tree obtained by identifying the central vertex of a star of order~$\frac{t}{2}$ with an endvertex of a path of order $\frac{t}{2}+2$. The~$\frac t2$-separator $z$ provided by  Lemma~\ref{lem:cut1} leaves exactly two components: a path of order~$\frac{t}{2}$ and a star of order $\frac{t}{2}$. One of the colour classes of this forest necessarily contains $\frac{3t}{4}-1$ vertices. 

However, it is possible to cut the tree at a different cutvertex so that the resulting forest admits a  significantly more balanced colouring than the one given by Lemma \ref{lem:cut4}. This is the purpose of Proposition~\ref{prop:cut5} below. Before we state the proposition, let us introduce some useful notation.

\begin{definition}
Given a graph $G$ and a proper $2$-colouring of its vertex set $c:V(G)\rightarrow \{0,1\}$ we define the imbalance of $c$ as $$\sigma(c) := |c_0|-|c_1|.$$ For a tree $T$ we will use $\sigma(T)$ to denote the imbalance of its unique $2$-colouring. 
\end{definition}

\begin{proposition}\label{prop:cut5}
Let $T$ be a tree on $t+1$ vertices. Then there exists $z\in V(T)$ and a proper $2$-colouring $c:V(T-z)\rightarrow\{0,1\}$ of $T-z$ with $|c_1|\leq |c_0|$ such that $|c_0|\leq \frac{2t}{3}$ and $|c_1|\leq \frac{t}{2}$.
\end{proposition}

\begin{proof}
We may assume that $t>3$. Assume the proposition does not hold, that is, for every $z\in V(T)$ and every proper $2$-colouring  of $T-z$, the heavier colour class of $T-z$ contains more than $\frac{2t}{3}$ vertices.

Let $z_0\in V(T)$ and $c:V(T-z_0)\rightarrow\{0,1\}$ as given by Lemma \ref{lem:cut4}. By our assumption above, we know that $c_0$, the heavier colour class induced by~$c$, contains between $\frac {2t}3$ and $\frac{3t-1}{4}$ vertices, while $c_1$, the lighter colour class, contains between $\frac t4$ and $\frac{t}{3}$ vertices. 

Consider the set $\{T_i\}_{i\in I}$ of all components of $T-z_0$. Let $J\subseteq I$ be the set of all indices $j$ such that $T_j$ has more vertices in $c_0$ than in $c_1$.  So clearly,

\begin{equation}\label{claim:Jbalance}
\text{$\sum_{j\in J}\sigma(T_j)>\frac{t}{3}$.}
\end{equation}

We claim that moreover,
	
\begin{equation}\label{claim:balance}
\text{for each $J'\subseteq J$ either \  $\sum_{j\in J'}\sigma(T_j)<\frac{t}{12}$ \ or \ $\sum_{j\in J'}\sigma(T_j)>\frac{t}{3}$.}
\end{equation}
	
Indeed, if this were not true for some $J'\subseteq J$, we could invert the colours in all trees in $\{T_j\}_{j\in J'}$. This yields a colouring with both colour classes having at most $\frac{2t}{3}$ vertices, because $c_0$ lost at least $\frac{t}{12}$ vertices, and $c_1$ gained at most $\frac t3$ vertices. This contradicts our assumption, and thus proves~\eqref{claim:balance}.

\smallskip

Call a family $J'\subseteq J$ {\em small} if  $\sum_{j\in J'}\sigma(T_j)<\frac{t}{12}$, and {\em large} otherwise (that is, by~\eqref{claim:balance}, if $\sum_{j\in J'}\sigma(T_j)>\frac{t}{3}$). It is easy to see that 
\begin{equation}\label{claim:starstar}
\text{if $J'$ is large, then there is a $j\in J'$ such that $\{j\}$ is large.}
\end{equation}
	
 In particular, we deduce from~\eqref{claim:balance} and~\eqref{claim:starstar} that there is an index $j^*\in J$ such that $\{j^*\}$ is large. That is, $\sigma(T_{j^*})>\frac{t}{3}$. Notice that there is only one such index
 as otherwise we could switch colour classes in one of the associated trees
 and obtain a contradiction to the initial assumption. So, by~\eqref{claim:starstar},
 \begin{equation}\label{claim:diamond}
\text{$J\setminus{\{j^*\}}$ is small.}
\end{equation}
 
We claim that
\begin{equation}\label{claim:balanced}
\text{$\sum_{i\in I\setminus{\{j^*\}}} \sigma(T_i) \leq \left\lceil\frac{t}{6}\right\rceil$.}
\end{equation}
Indeed, otherwise we could switch colours in every tree $T_j$ with  $j\in J\setminus{\{j^*\}}$ to obtain a new colouring of $T-z_0$. In this new colouring, $c_1$ has gained less than $\frac t{12}$ vertices (by~\eqref{claim:diamond}), and thus stays below $\frac t2$. Moreover, $c_1$ now contains the larger bipartition class of every tree $T_i$ with $i\in I\setminus{\{j^*\}}$, and thus has at least
$$ \sum_{i\in I\setminus{\{j^*\}}}\left(\frac{|T_i|}{2}+\frac{\sigma(T_i)}{2}\right)>\frac{t-|T_{j^*}|}{2} + \frac{\lceil\frac{t}{6}\rceil}{2} \geq \frac{t-1}{4}+\frac{t+3}{12} = \frac{t}{3}$$
vertices.
Therefore, $c_0$ now has less than $\frac{2t}{3}$ vertices, a contradiction to our assumption that no such colouring exists. This proves~\eqref{claim:balanced}.\\
	
Now, apply Lemma \ref{lem:cut1} to obtain $z_1\in V(T_{j^*})$ such that every component of $T_{j^*}-z_1$ covers at most $\lceil\frac{|T_{j^*}|-1}{2}\rceil$ vertices. Let $T_{z_0}$ denote the component of $T-z_1$ that contains $z_0$ and let~$\{C_{\ell}\}_{\ell\in L}$ denote the set of all other components of $T-z_1$. Further, 
let~$C_{z_o}$ denote the unique component of $T_{j^*}-z_1$ that is contained in $T_{z_0}$, if such a component exists. Observe that Lemma \ref{lem:cut1} allows us to assume that 
\begin{equation}\label{sizeCz0}
|C_{z_o}|\leq \left\lfloor\frac{|T_{j^*}|-1}{2}\right\rfloor  \leq \left\lfloor\frac{\lceil\frac t2\rceil-1}{2}\right\rfloor \leq \frac{t-1}4.
\end{equation}

Next, group the elements of $\{C_{\ell}\}_{\ell\in L}$ into  two forests $F^A$ and~$F^B$ fulfilling 
\begin{equation}\label{claim:FAFB}
\text{ $\max\{|F^A|,|F^B|\}\leq \frac{t+1}{3}$,}
\end{equation}
which is possible by Lemma \ref{lem:num}~(ii), and since $\max\{\frac{|T_{j^*}|-1}{2},\frac{2}{3}|T_{j^*}|\}\leq \frac{t+1}{3}$.

For $i=A,B$, consider the proper $2$-colouring $c^i$ induced by $T_{j^*}$ on $F^i$. 
By symmetry, we may assume that
\begin{equation}\label{sigmaAsigmaB}
\text{$\sigma(c^A)\geq \sigma(c^B)$.}
\end{equation}

Observe that by~\eqref{claim:Jbalance}, and by the choice of $c^i$, we have that
$$\frac t3<\sigma(T_{j^*}) \leq \sigma_{C_{z_0}\cup\{z_1\}}+ \sigma(c^A) + \sigma(c^B),$$ 
where $\sigma_{C_{z_0}\cup\{z_1\}}$ denotes the imbalance that $T_{j^*}$ induces on $C_{z_0}\cup\{z_1\}$. Note that $\sigma_{C_{z_0}\cup\{z_1\}}\leq \max\{|C_{z_o}|,1\}$, Therefore, 
	
\begin{equation}\label{claim:unbalanced2}
	\sigma(c^A) + \sigma(c^B) > \frac{t}{3}-\max\{|C_{z_o}|,1\}. 
\end{equation}
	
Now we consider and separately treat  two possible cases, according to the imbalance of the canonical colouring of  $T_{z_0}$. For convenience, let $A(T_{z_0})$ denote the larger colour class of $T_{z_0}$ in this colouring, and let  by $B(T_{z_0})$ denote the smaller colour class.\\
	
\noindent\textbf{Case 1: }$\sigma(T_{z_0})\leq \frac{t}{3}$.\\
	
In this case, define a new colouring $c'$ by setting $c_0':=A(T_{z_0})\cup c_1^A\cup c_0^B$ and $c_1':=(T-z_1)\setminus{c_0'} = B(T_{z_0})\cup c_0^A\cup c_1^B$. Then, by~\eqref{sigmaAsigmaB},
$$|c_0'|=\frac{|T_{z_0}|}{2}+\frac{\sigma(T_{z_0})}{2}+\frac{|F^A|}{2}-\frac{\sigma(c^A)}{2}+\frac{|F^B|}{2}+\frac{\sigma(c^B)}{2} \leq \frac{t}{2} + \frac{\sigma(T_{z_0})}{2} \leq \frac{2t}{3}, $$
and moreover, by~\eqref{claim:FAFB},

$$|c_1'|< \frac{|T_{z_0}|}{2} + \max\{|F^A|-1,1\} + \frac{|F^B|}{2}  \leq \frac{2t}{3},$$
and hence after possibly swapping colours we found a colouring as desired for the proposition (with $z_1$ in the role of $z$). This is a contradiction, since we assumed no such colouring exists.\\

\noindent\textbf{Case 2: }$\sigma(T_{z_0})> \frac{t}{3}$.\\
	
This time we define $c'$ by setting $c_0':=A(T_{z_0})\cup c_1^A\cup c_1^B$ and $c_1':=B(T_{z_0})\cup c_0^A\cup c_0^B$. Let $\sigma_{C_{z_0}\cup\{z_0\}}$ denote the imbalance that $T_{z_0}$ induces on $C_{z_0}\cup\{z_0\}$
and note that by~\eqref{sizeCz0}, we have that $$\sigma_{C_{z_0}\cup\{z_0\}}\leq \max\{|C_{z_o}|,1\}\leq\frac{t-1}4.$$ Recalling~\eqref{claim:balanced} and~\eqref{claim:unbalanced2}, we obtain
\begin{align*}
|c_0'| & =\frac{t}{2}+\frac{\sigma(T_{z_0})-(\sigma(c^A)+\sigma(c^B))}{2} \\
& \leq \frac{t}{2} + \frac{\sum_{i\in I\setminus{\{j^*\}}}\sigma(T_i)+\sigma_{C_{z_0}\cup\{z_0\}}+\max\{|C_{z_o}|,1\}-\frac{t}{3}}{2} \\
& \leq \frac{t}{2}  + \frac{t+2}{12} +\max\{|C_{z_o}|,1\} - \frac{t}{6} \\
& \leq \frac{2t}{3}.
\end{align*}
Furthermore, by~\eqref{claim:FAFB},
$$|c_1'| \leq \frac{|T_{z_0}|}{2}-\frac{\sigma(T_{z_0})}{2} + |c_0^A|+|c_0^B| \leq \frac{2t}{3},$$
We thus again obtain a contradiction.
\end{proof}

\section{Embedding trees in robust components}\label{sec:proof1}
In this section, we discuss the embedding of trees into a large robust component of some host graph, by which we mean we embed into graphs whose corresponding reduced graph has a large connected component. The arguments depend on whether  the reduced graph is bipartite or not, and hence we deal with these situations separately. 

The main results from this section are Propositions~\ref{prop:bipartite} and~\ref{prop:non bipartite} and their corollaries. They  will be used in the proof of Proposition~\ref{prop:connected-cte}, our main embedding result for robust components. Moreover, they will be one of the tools in the proof of our key embedding lemma, Lemma~\ref{lem:superlemma}, on which most
 of our main results rely.

\subsection{The bipartite case}
As we mentioned in the introduction, any tree with $k$ edges greedily embeds in any graph of minimum degree at least~$k$. In a bipartite graph $H=(X,Y;E)$ the minimum degree condition can be relaxed: If the tree $T$ has bipartition classes of sizes $k_1$ and $k_2$, then
it is clearly enough to require the vertices from $X$ to have degree at least $k_1$ and the vertices from $Y$ to have degree at least $k_2$. In particular, if $\deg(x)\ge \lfloor\frac{k}{2}\rfloor$ for all $x\in X$, and  $\deg(y)\ge k$ for all $y\in Y$, then each tree with~$k$ edges embeds in $H$. 

If the tree we wish to embed has bounded degree, and the  host graph has an $(\varepsilon,\eta)$-reduced graph which is bipartite and connected, for some $\eps$ and $\eta$, one can do even better: We will now show that in this case it is enough to  require a minimum degree of roughly $\frac k2$ for the vertices in only one of the bipartition classes, as long as this class is not too small.

\begin{proposition}\label{prop:bipartite}
For all $\varepsilon\in (0,10^{-8})$ and for all $d, M_0\in\NN$, there is $k_0$ such that for all $n,k\geq k_0$ 
  the following holds. Let $G$ be an $n$-vertex  graph,  with $(\varepsilon,5\sqrt\varepsilon)$-reduced graph  $R$ that satifies $|R|\le M_0$, such that
\begin{enumerate}[(i)]
\item $R=(X,Y)$ is bipartite and connected;\label{bip_con}
\item $\diam(R)\leq d$;\label{diameter}
\item  $\deg(x)\ge (1+100\sqrt\varepsilon)\frac{k}{2}\cdot \frac{|R|}n$, for all $x\in X$; and \label{mindegX}
\item  $|X|\ge (1+100\sqrt{\varepsilon})k\cdot \frac{|R|}n$.\label{sizeX}
 \end{enumerate}
Then $G$ contains every tree $T$ with $k$ edges and $\Delta(T)\le k^{\tfrac{1}{d}}$ as a subgraph. 
\end{proposition}
	
\begin{proof}
Given $\eps,d$ and $M_0$ as in the Theorem, we set
$$k_0:= \left( \frac{8M_0^2}{\eps^2} \right)^d.$$

Let $G$ be a graph as in Proposition~\ref{prop:bipartite}, let $X_1\cup\dots\cup X_s$ and $Y_1\cup\dots\cup Y_t$ be the $(\varepsilon,5\sqrt{\varepsilon})$-regular partition of $G$  corresponding to the reduced graph $R$ (in particular $s+t\le M_0$). Set  $m:=|X_i|=|Y_j|$ (for any $i,j$). 

For each $i\in[s]$, we arbitrarily partition $X_i$ into three sets $X_{i,S}$, $X_{i,L}$, $X_{i,C}$; and for each $j\in[t]$ we arbitrarily partition~$Y_j$ into three sets $Y_{j,S}$, $Y_{j,L}$, $Y_{j,C}$, such that 
$$|X_{i,S}|=|X_{i,L}|=
|Y_{j,S}|=|Y_{j,L}|=\lceil10\sqrt{\varepsilon} m\rceil .$$
The letters $S$, $L$ and $C$ stand for seeds, links and clusters, respectively (sets $X_{i,C}$ and $Y_{j,C}$ contain the bulk of the clusters). We also call these subsets the $L$-, $S$- or $C$-slice of the corresponding cluster.

Note that, by Fact~\ref{fact:1}, for every $(X_i,Y_j)$ with positive density, each of the pairs $(X_{i,K},Y_{j,K'})$, with $K,K'\in\{S,L,C\}$, is $\frac{\sqrt{\varepsilon}}{5}$-regular with density greater than $4\sqrt{\varepsilon}$.

Root $T$ at any vertex $r(T)$. By Proposition~\ref{prop:cut3}, with parameters $\beta=\frac\varepsilon{s+t}$, we obtain a decomposition of  $T$ into a collection of pieces $\mathcal{P}$, each of order at most $\beta k$, and a family of seeds $S$ of size at most $\frac 2\beta$. Order the elements from $S\cup\mathcal P$ in a way that the first element is $r(T)$, and the parent of each element is either an earlier seed or belongs to an earlier piece.

\smallskip

Our plan is to embed the  elements from $S\cup\mathcal P$ in this order. Seeds will go to $S$-slices of appropriate clusters $X_{i,S}$ or $Y_{j,S}$, with $r(T)$ going to cluster $X_i$ if $r(T)$ belongs to the heavier bipartition class of $T$,
and going to $Y_j$ otherwise.
Pieces from $\mathcal P$ will go into $C$-slices $(X_{i,C},Y_{j,C})$ of appropriate pairs $(X_i,Y_j)$, and into $L$-slices of other clusters. 

More precisely, for each piece $P\in \mathcal P$ we will find a pair  $(X_i,Y_j)$ such that there is enough space left in $(X_{i,C},Y_{j,C})$ to accommodate $P$. At this point, the parent of $P$ is already embedded into some cluster $Z$, so we need to embed part of $P$ into a path $ZZ_1Z_2Z_3\ldots Z_h$ that connects $Z$ with the pair $(X_{i},Y_{j})$. Because of the bounded degree of $T$, and since the diameter of $G$ is also bounded, this path can be chosen short enough to ensure that the levels of $P$ that are embedded into this path only contain relatively few vertices. So we can use the $L$-slices of the clusters $Z_{\ell}$ for these levels. The remaining levels of $P$ will be embedded into the free space of $(X_{i,C},Y_{j,C})$.

\smallskip

Let us make this sketch more precise. 
During the embedding procedure, we will write $X'_{i,C}$ and $Y'_{j,C}$ for the set of unoccupied vertices of $X_{i,C}$ and $Y_{j,C}$ respectively. We will say that a pair $(X_i,Y_j)$ is \textit{good} if $d(X_i,Y_j)>0$ and $\min\{|X'_{i,C}|,|Y'_{j,C}|\}\ge 5\sqrt\varepsilon m$. Hence we will be able to apply  Lemma \ref{lem:T1} to any good pair and any piece belonging to $\mathcal{P}$. 

The embedding $\phi:V(T)\to V(G)$ will be constructed iteratively, following the embedding order of $S\cup\mathcal P$ chosen above. Employing the strategy explained above, we make sure that at every step, the following conditions will be satisfied: 

\begin{enumerate}[(A)]
\item\label{bip:0} Each vertex is embedded into a neighbour of the image of its already embedded parent;
	\item\label{bip:1} each $s\in S$ is embedded into the $S$-slice of some cluster;
	\item\label{bip:2} for each $P\in  \mathcal P$, the first (up to $d-1$) levels are embedded into the $L$-slices of some clusters, and the rest goes into the $C$-slices;
	 and
	\item\label{bip:3} every  $v\in V(T)$ is mapped into a vertex that is typical towards both the $S$-slice and the $L$-slice of some adjacent cluster. 
\end{enumerate}

Since the set $S$ has constant size, and since we do not particularly care into which cluster a seed goes, as long as it goes to the $S$-slice, it is clearly possible to embed a seed $s$, when its time comes, satisying conditions~\eqref{bip:0}, ~\eqref{bip:1} and~\eqref{bip:3}.

So assume we are about to embed a piece $P\in \mathcal P$. The parent of the root $r(P)$ of $P$ is already embedded into some vertex that is typical with respect to the $L$-slice of some cluster $Z_1$. 
In order to be able to embed $P$ according to our plan, it suffices to ensure that
	
\begin{enumerate}[(I)]
\item\label{claim:1} there exists some good pair $(X_i,Y_j)$;
\item \label{condi_path} there is a path $Z_1Z_2Z_3\ldots Z_h$ of length $h\leq d$ from $Z_1$ to $X_i$;
\item\label{condi_Lfits} the union of the first $h-1$ levels of $P$ is small enough to fit into the free space in the $L$-slices of $\{Z_1, Z_2, Z_3, \ldots , Z_{h-1}\}$.
\end{enumerate}

If we can assure these properties, we can repeatedly apply Lemma~\ref{lem:T1} to embed the first levels of $P$ into the $L$-slices of the clusters $Z_\ell$, and  the remaining levels of $P$ into $(X'_{i,C},Y'_{j,C})$ in a way that ~\eqref{bip:0}, ~\eqref{bip:2} and~\eqref{bip:3} hold.

\smallskip

So, let us  prove~\eqref{claim:1}.
We first note that there exists some cluster $X_i$ such that $|\phi^{-1}(X_{i,C})|<|X_{i,C}|-5\sqrt{\varepsilon}m$. Indeed, otherwise we have used at least
$$(1-25\sqrt{\varepsilon})|X|-5\sqrt{\varepsilon}|X|\ge (1-30\sqrt{\varepsilon})(1+100\sqrt{\varepsilon})k>(1+2\sqrt{\varepsilon})k>k+1$$
vertices from $X$ already, a contradiction, since $|T|=k+1$. 
 
Next, we claim there exists some cluster $Y_{j}$ such that $(X_i,Y_j)$ is good. If this was not the case, then we have used at least
$$(1-30\sqrt{\varepsilon})|N(X_i)|\ge (1-30\sqrt{\varepsilon})(1+100\sqrt{\varepsilon})\frac{k}{2} >(1+2\sqrt{\varepsilon})\frac{k}{2} > \frac{k+1}2$$
vertices of $Y$ already, a contradiction, as we placed the root $r(T)$ of $T$ in a way that guaranteed we would embed the smaller bipartition class of $T$ into~$Y$.

Observe that~\eqref{claim:1} implies~\eqref{condi_path}, because of condition~\eqref{diameter} of Proposition~\ref{prop:bipartite}. So it only remains to prove~\eqref{condi_Lfits}.

Using \eqref{bip:2} for already embedded pieces $P'$, and using  the fact that, for any such piece $P'$, the number of vertices in their first $d-1$ levels is bounded by $2(\Delta(T)-1)^{d-2}$ (except if $\Delta(T)\leq 2$, in which case this number is bounded by~$d-1$),
we have that the total number of occupied vertices in $L$-slices  is at most 
$$|S|\cdot \Delta(T)\cdot 2 (\Delta(T)-1)^{d-2} \ \leq \ \frac 4\beta \cdot k^{\frac{d-1}{d}} \ \le \frac {4M_0}\eps \cdot k^{\frac{d-1}{d}}\ < \eps \frac{k}{2M_0}\ \le \ \varepsilon m$$ for $k\ge k_0$. In particular, each $L$-slice of a cluster $Z_{\ell}$ has at least $\lceil9\sqrt{\varepsilon} m\rceil$ unused vertices.
This is enough to ensure that the first $h-1$ levels of $P$ fit  into the $L$-slices of the clusters $Z_1$, $Z_2$, $Z_3, \ldots , Z_{h-1}$. This proves~\eqref{condi_Lfits}.
\end{proof}

\begin{remark}\label{rem:bipartite}
It is easy to see that instead of conditions~\eqref{mindegX} and~\eqref{sizeX} from Proposition~\ref{prop:bipartite}  we could use the weaker requirement that there is a set $\mathcal{C}$ of clusters in $X$ such that $\deg(x)\ge (1+100\sqrt\eps)\frac{k}{2}\cdot \frac{|R|}n$, for all $x\in V(\mathcal{C})$, and $|V(\mathcal{C})|\ge (1+100\sqrt \eps)k\cdot \frac{|R|}n$. 
\end{remark}

\begin{remark}\label{rem:bipartite2}
Observe that  Proposition~\ref{prop:bipartite} remains true with the following additional conditions. Let $U\subseteq V(G)$ such that 
\begin{itemize}
\item $|U|+|T|\leq k+1$;
\item $|U\cap \bigcup X|+c_0(T)\leq k$; and 
\item $|U\cap \bigcup Y|+c_1(T)\leq \frac{k}{2}$,
\end{itemize}
where $c_0(T)$ and $c_1(T)$ are the two colour classes of $T$.\\
Then $T$ can be embedded into $G$ avoiding $U$, that is, we can embed $T$ in such a way $\phi(V(T))\subset V(G)\setminus U$.  
\end{remark}

Moreover, observe that by repeatedly applying Proposition~\ref{prop:bipartite} together with Remark~\ref{rem:bipartite2}, we can actually embed a forest instead of a tree.
We say that a forest $F$, with colour classes $C_1$ and $C_2$, is a {\em $(k_1,k_2,t)$-forest } if
\begin{enumerate}
	\item $|C_i|\leq k_i$ for $i=1,2$, and
	\item $\Delta( F)\leq (k_1+k_2)^{t}$.\end{enumerate}

\begin{corollary}\label{emb:forest}
For all $\varepsilon\in (0,10^{-8})$ and for all $d,M_0\in\NN$ there is $k_0$ such that for all $n,k_1, k_2\geq k_0$ 
 the following holds. 
Let $G$ be an $n$-vertex having a $(\varepsilon,5\sqrt\varepsilon)$-reduced graph $R$ that satisfies $|R|\le M_0$, such that
\begin{enumerate}[(i)]
\item $R=(X,Y)$ is connected and bipartite;
\item $\diam(R)\leq d$;
\item $\deg(x)\ge (1+100\sqrt\varepsilon)k_2 \cdot \frac{|R|}n$, for all $x\in X$; and
\item $|X|\ge (1+100\sqrt\varepsilon)k_1 \cdot \frac{|R|}n$.
\end{enumerate}
Then any $(k_1,k_2,\frac 1d)$-forest $ F$, with colour classes $C_1$ and $C_2$, can be embedded into $G$, with $C_1$ going to $\bigcup X$ and $C_2$ going to $\bigcup Y$. \\
Moreover, if ${F}$ has at most $\tfrac{\eps n}{|R|}$ roots, then the images of the roots going to $\bigcup X$ can be mapped to any prescribed set of size at least $2\varepsilon|\bigcup X|$ in $X$, and the images of the roots going to $\bigcup Y$ can be mapped to any prescribed set of size at least $2\varepsilon|\bigcup Y|$ in $Y$.
\end{corollary}

\begin{remark}\label{rem:bipartite2F}
An analogue of Remark~\ref{rem:bipartite2} holds for the situation of Corollary~\ref{emb:forest}.
\end{remark}

\smallskip

It is easy to see that we can bound the balancedness of trees whose maximum degree is bounded by some constant $\Delta$. See Section~\ref{sec:const} for an example of the most unbalanced such tree.

So it comes as no surprise that for the class of all constant degree trees, it is possible to
 show the following improvement of Proposition~\ref{prop:bipartite}. 

\begin{corollary}\label{cor:bipartite_bounded}
For all $\varepsilon\in (0,10^{-8})$, $d,M_0\in\NN$ and $\Delta\geq2$ there is a~$k_0$ such that for all $n,k\geq k_0$ 
 the following holds. 
Let $G$ be an $n$-vertex  graph that has an $(\varepsilon,5\sqrt{\varepsilon})$-reduced graph $R$ that satisfies $|R|\le M_0$, such  that
\begin{enumerate}[(i)]
\item $R=(X,Y)$ is connected and bipartite;
\item  $\diam(R)\leq d$;
\item $\deg(x)\ge (1+100\sqrt\varepsilon)\frac{k}{2}\cdot \frac{|R|}n$ for all $x\in X$;
\item $|X|\ge (1+100\sqrt{\varepsilon})\frac{(\Delta-1)}{\Delta}k\cdot \frac{|R|}n$.
\end{enumerate}
Then $G$ contains every tree $T$ with $k$ edges and $\Delta(T)\le \Delta$ as a subgraph. 
\end{corollary}
	
\subsection{The nonbipartite case}	

In this section we treat tree embeddings into graph with large nonbipartite components in the reduced graph. 
The proof of the corresponding proposition, Proposition~\ref{prop:non bipartite} below, 
is  very similar to the proof of Proposition~\ref{prop:bipartite}. For convenience, we will now work with a matching in the reduced graph.

 Let $G$ be a graph with an $(\varepsilon,\eta)$-regular partition, we say that $M$ is a cluster matching if it is a matching in the corresponding $(\varepsilon,\eta)$-reduced graph. We begin our treatment of the nonbipartite case by showing that we can always find a large cluster matching in graphs with large minimum degree that admit an $(\varepsilon,\eta)$-regular partition, for some $\eps, \eta\in (0,1)$.

\begin{lemma}\label{lem:matching2}
	Let $\varepsilon,\eta\in(0,1)$ and  let $t,\ell\in\mathbb{N}$. Let $G$ be a graph on $n\ge 2t$ vertices and minimum degree at least $t$, such that $G$ admits an $(\varepsilon,\eta)$-regular partition with less than $\ell$ parts. Then there exists an $(5\varepsilon,\eta-\varepsilon)$-regular partition of $G$ with less than $2\ell$ parts, a cluster matching $M$ and an independent family of clusters~$\mathcal{C}$, disjoint from~$M$, such that	
	\begin{enumerate}[(i)]
		\item $V(\bigcup M)\cup V(\bigcup\mathcal{C})=V(G)$;\label{cover}
		\item $|\bigcup M|\geq 2t$; and\label{size}
		\item there is a  partition $V(M)=V_1\cup V_2$, such that $N(\mathcal{C})\subseteq V_1$  and  every edge in $M$ has one endpoint in $V_1$ and one endpoint in $V_2$.\label{partition}
	\end{enumerate}
\end{lemma}
\begin{proof} Let $R$ be the reduced graph corresponding to the $(\varepsilon,\eta)$-regular partition of $G$. By applying Lemma~\ref{lem:matching} to $R$, we obtain an independent set~$I$, a matching $M'$ and a set of disjoint triangles $\Gamma$, such that $V(R)=I\cup V(M')\cup V(\Gamma)$. If $\Gamma$ is empty, we are done by choosing $\mathcal{C}:=I$ and $M:=M'$. So suppose $\Gamma\neq\emptyset$. 
	
	We arbitrarily partition each cluster $X\in V(R)$ into $X^1$ and $X^2$ so that $\big||X^1|-|X^2|\big |\le 1$. Thanks to Fact \ref{fact:1,2}, the partition $V(G) = \bigcup_{X\in V(R)}X^1\cup X^2$ is $(5\varepsilon,\eta-\varepsilon)$-regular and has less than $2\ell$ atoms. We set
	
	$$M:=\bigcup_{CD\in M'}\{(C^1D^1),(C^2D^2)\}\cup\bigcup_{XYZ\in\Gamma} \{(X^1,Y^2),(Y^1,Z^2),(Z^1,X^2)\}$$
	
	\noindent and
	
	$$\mathcal{C}:=\bigcup_{C\in I}\{C^1,C^2\}$$
	
	Note that $\mathcal{C}$ and $M$ inherit the properties of $I$ and $M'$, respectively. Property~\eqref{size} follows from Property~\eqref{partition} and the minimum degree of $G$.
\end{proof}		

Let us note that in Lemma~\ref{lem:matching2}, given $\delta>0$, if $G$ has minimum degree at least $(1+\delta)\tfrac{k}{2}$, then one can find a matching covering at least $(1+\delta)k$ vertices of $G$. 

\smallskip

Now we are ready for Proposition~\ref{prop:non bipartite} and its proof.

\begin{proposition}\label{prop:non bipartite}
	For all $\varepsilon\in (0,10^{-8})$ and $d,M_0\in\NN$ there exists $k_0$ such that for all $n,k\geq k_0$ 
	the following holds. 
Let $G$ be an $n$-vertex  graph that has an $(\varepsilon,5\sqrt{\varepsilon})$-reduced graph $R$ that satisfies $|R|\le M_0$, such  that
\begin{enumerate}[(i)]
\item $R$ is connected and nonbipartite;
\item  $\diam(R)\leq d$; and
	\item $R$ has a matching $M$ with $|V(M)|\ge (1+100\sqrt{\varepsilon})k\cdot \frac{|R|}n$.
	\end{enumerate}
	 Then $G$ contains every tree $T$ with $k$ edges and $\Delta({T})\le k^{\frac{1}{3d+1}}$  as a subgraph. 
\end{proposition}	

\begin{proof}
Given $\eps, d$ and $M_0$ as in the Theorem, we set
$$k_0:= \left(\frac{8M_0^2}{\eps^2}\right)^{3d+1}$$

Now, let $G$ be a graph as in Proposition~\ref{prop:non bipartite}, let $V(G)=V_1\cup\ldots\cup V_{\ell}$ be the $(\varepsilon,5\sqrt\varepsilon)$-upper regular partition of $G$ corresponding to the reduced graph $R$(in particular $\ell\le M_0$). Set $m:=|V_i|$ for any $i\in[\ell]$. 

For each $i\in[\ell]$, we partition cluster  $V_i$ into sets $V_{i,S}$, $V_{i,L}$, $V_{i,C}$ in the same way as we did in Proposition~\ref{prop:bipartite}. Also, consider the decomposition of $T$ into~$\mathcal{T}$ and $S$ given by Proposition~\ref{prop:cut3}, with $\beta=\frac{\eps}{\ell}$. We order $S\cup\mathcal P$ in the same way as in the proof of Proposition~\ref{prop:bipartite}.
	
The embedding $\phi:V(T)\to V(G)$ will be constructed iteratively, following the order of $S\cup\mathcal P$. 
			We make sure that at every step, the following conditions will be satisfied: 
	
	\begin{enumerate}[(A)]
		\item\label{nbip:0} Each vertex is embedded into a neighbour of the image of its already embedded parent;
		\item\label{nbip:1} each $s\in S$ is embedded into the $S$-slice of some cluster;
		\item\label{nbip:2} for each $P\in  \mathcal P$, the first $3d$ levels are embedded into the $L$-slices of some clusters, and the remaining levels go to the $C$-slices;
		\item\label{nbip:3} every  $v\in V(T)$ is mapped to a vertex that is typical with respect to both the $S$-slice and the $L$-slice of some adjacent cluster; and
		\item\label{nbip:4} $\big||\phi^{-1}(V_{i,C})|-|\phi^{-1}(V_{j,C})|\big|\le\epsilon m$ for each pair $(V_i,V_j)\in M$. 
	\end{enumerate}
	
	We already know that it is no problem to embed a seed $s$, when its time comes, satisfying conditions~\eqref{nbip:0}, ~\eqref{nbip:1} and~\eqref{nbip:3}. So we mainly have to worry about~\eqref{nbip:2} and~\eqref{nbip:4}.
	
	Assume we are about to embed a piece $P\in \mathcal P$. The parent of the root $r(P)$ of $P$ is already embedded into some vertex that is typical with respect to the $L$-slice of some cluster~$Z_1$. In order to be able to embed $P$ so that the above conditions are satisfied, it suffices to ensure that
	
	\begin{enumerate}[(I)]
		\item\label{nclaim:1} there exists some good pair $(V_i,V_j)$;
		\item \label{ncondi_path} for either choice of $Z_{3d+1}\in\{V_i, V_j\}$ there is a walk $Z_1Z_2\ldots Z_{3d+1}$ in~$R$;
		\item\label{ncondi_Lfits} the first $3d$ levels of $P$ are small enough to fit into the free space in the $L$-slices of $\{Z_1, Z_2\ldots , Z_{3d}\}$,
	\end{enumerate}
	where a walk in a graph is a sequence $Z_1Z_2\ldots Z_{h}$ such that each $Z_i$ is adjacent to $Z_{i+1}$ for all $1\le i<h$.
	
	Before we prove~\eqref{nclaim:1}--\eqref{ncondi_Lfits}, let us explain why these conditions are enough to ensure we can embed $T$ correctly. As before, we plan to repeatedly apply Lemma~\ref{lem:T1} in order to embed the first levels of $P$ into the $L$-slices of the clusters $Z_1, Z_2, \ldots , Z_{3d}$, and the later levels into the $C$-slices of $V_i,V_j$, always avoiding  all vertices used earlier. 
	
	Since our aim is to embed $P$ in a way that~\eqref{nbip:4} is fulfilled, we take care to choose $Z_{3d+1}\in\{V_i, V_j\}$  in a way that the larger bipartition class of the tree~$P'$ obtained from $P$ by deleting its first $3d$ levels goes to the less occupied  slice $V_{i,C}$, $V_{j,C}$. That is, assuming that $\big |\phi^{-1}(V_{i,C})\big|\le \big|\phi^{-1}(V_{j,C})\big|$ (the other case is analogous), we proceed as follows. If the levels of $P'$ that lie at even distance from the root of $P$ in total contain more vertices than those lying at odd distance, we choose $Z_{3d+1}=V_j$. Otherwise, we choose $Z_{3d+1}=V_i$.
	We then embed $P$, making the first $3d$ levels go to $L$-slices, and embedding~$P'$ into  $V_{i,C}\cup V_{j,C}$. 
		
		\medskip
		
	Let us now prove~\eqref{nclaim:1}. Suppose there is no good pair in $M$. This together with~\eqref{nbip:4} implies that the number of embedded vertices is at least
	$$\sum_{AB\in M}(|A_{i,C}|-6\sqrt{\eps}m+|B_{i,C}|-6\sqrt{\eps}m)\geq (1-33\sqrt{\eps})(1+100\sqrt{\eps})k>k+1,$$ 
	a contradiction, since $|T|=k+1$. 	
	
	Next, we show~\eqref{ncondi_path}. Assume we chose $Z_{3d+1}=V_i$ (the other case is analogous). Let $C=C_1C_2\ldots C_pC_1$ be a minimal odd cycle in the reduced graph. Since $C$ is minimally odd, the shortest path between two clusters in $C$ is the shortest arc in the cycle, and hence $p\le 2d+1$. Let $U:=Z_1U_1\ldots U_sC_1$ be a shortest path  from $Z_1$ to $C_1$ and let $Q:=C_{\lceil\frac p2\rceil} Q_1\ldots Q_tV_i$ be a shortest path from $C_{\lceil\frac p2\rceil}$ to $V_i$. As $\diam(R)\le d$, we have that  $s+t+2\le 2d$. So, by using the appropriate one of the two $C_1$--$C_{\lceil\frac p2\rceil}$ paths in $C$, we can extend $U\cup Q$ to an odd walk of  length at most $2d+(d+1)=3d+1$, which connects $Z_1$ with~$V_i$. By going back- and forwards on this walk, if necessary, we can obtain a walk of length exactly $3d+1$, which is as desired. So, condition~\eqref{ncondi_path} holds.
	
	Finally, using the same reasoning as in Proposition~\ref{prop:bipartite} we can prove that the total number of occupied vertices in $L$-slices is at most 
	$$|S|\cdot \Delta(T)\cdot 2(\Delta(T)-1)^{3d-1} \ \leq \ \frac 4\beta \cdot k^{\frac{3d}{3d+1}} \ < \ \varepsilon m$$ for $k\ge k_0$. In particular, the $L$-slice of each cluster has at least $\lceil9\sqrt{\varepsilon} m\rceil$ unused vertices and, therefore, we can embed each vertex of the first $3d$ levels of $P$ into the $L$-slices of the clusters from the walk $Z_1Z_2\ldots Z_{3d}$ without a problem. This proves~\eqref{ncondi_Lfits}.

\end{proof}	

\begin{remark}
	If $d=1$ we can actually embed trees with maximum degree bounded by $\rho k$, where $\rho$ is a sufficiently small constant, without modifying our proof significantly, because we can reach both $V_i$ and $V_j$ in one step from the image of the latest embedded seed.
\end{remark}

\begin{remark}\label{rem:nbipartite2}
	Similar as in the bipartite case, we can add an extra hypothesis as in Remark~\ref{rem:bipartite2}. Consider an arbitrary set $U\subset V(G)$ such that $|U|+|T|\leq k+1$ and such that $U$ is reasonably balanced in $M$, that is, $\big||U\cap C|-|U\cap D|\big|<\varepsilon |C|$ for all $CD\in M$. Then $T$ can be embedded into $G$ avoiding $U$.  
\end{remark}

Repeatedly applying Proposition~\ref{prop:non bipartite}, together with Remark~\ref{rem:nbipartite2}, we can embed a forest instead of a tree.

\begin{corollary}\label{emb:forest2}
	Let $\varepsilon\in (0,10^{-8})$ and let $d,M_0\in\NN$. There exists $k_0\in \NN$ such that for all $n,k\geq k_0$ 
 the following holds. Let $G$ be a $n$-vertex graph with an $(\varepsilon,5\sqrt{\varepsilon})$-reduced graph $R$ that satisfies $|R|\le M_0$. Suppose that
	\begin{enumerate}[(i)]
	 \item $R$ is connected and nonbipartite;
	 \item $\diam(R)\le d$;
	 \item $R$ has a matching $M$ with $|V(M)|\ge (1+100\sqrt{\varepsilon})k\cdot \frac{|R|}n$;
	 \end{enumerate}
	 then any forest ${F}$ on at most $k+1$ vertices that satisfies $\Delta({{F}})\le k^{\frac{1}{3d+1}}$ is a subgraph of $G$. 
	Moreover, if ${F}$ has at most $\tfrac{\eps n}{|R|}$ roots, then the images of the roots can be mapped into any prescribed set of
	size at least $2\epsilon n$.
\end{corollary}	


\section{Improving the maximum degree bound}\label{improve}

In the previous section we proved that  in graphs of minimum degree at least $(1+\delta)\tfrac{k}{2}$ which have a large connected component  after applying regularity and performing the usual cleaning-up, all trees of maximum degree  $k^{O(\delta)}$ appear as subgraphs (see Theorem~\ref{thm:Erdos}). The aim of the present section is to prove a similar statement as there, but with a significant
weakening in the bound on the maximum degree of the tree. More precisely, the exponent in this bound will no longer depend on the diameter of the reduced graph. 

The exact statement is given in Proposition~\ref{prop:connected-cte}, and its proof can be found  in Subsection~\ref{subsec:improve}. Then, in Subsection~\ref{sec:ES}, we give a quick application of  Proposition~\ref{prop:connected-cte} to the Erd\H os-S\'os conjecture.

\subsection{Improving the exponent from the maximum degree bound}\label{subsec:improve}

We need a theorem from~\cite{ERDOS1989}, which says that one can bound the diameter of any connected graph in terms of its number of vertices and its minimum degree. 
\begin{theorem}[Erd\H{o}s, Pach, Pollach and Tuza~\cite{ERDOS1989}]\label{thm:Erdos} Let $G$ be a connected graph on $n$ vertices with minimum degree at least $2$. Then $$\diam(G)\leq \left\lfloor\tfrac{3n}{\delta(G)+1}\right\rfloor - 1.$$
\end{theorem}

We also need the following lemma.  Given a graph $G$ and a vertex $v\in V(G)$, let $N_i(v)$ denote the $i^{\text{th}}$ neighbourhood of $v$ (i.e.~the set of vertices of $G$ at distance $i$ from $v$). 

\begin{lemma}\label{lem:dist} 
	Let $q\in\NN$ and let $G$ be a connected graph, and let $v\in V(G)$. Then
	$$\left|\bigcup_{i=0}^{3q+1}N_i(v)\right| \geq \min\{(q+1)(\delta (G)+1),|V(G)|\}.$$
\end{lemma}

\begin{proof}
	 If $N_i(v)=\emptyset$ for some $i\in[3q+1]$, then, as $G$ is connected, $V(G)\subseteq \bigcup_{j=0}^{i-1}N_j(v)$ and thus $|\bigcup_{j=0}^{3q+1}N_j(v)| = |V(G)|$. Therefore, we assume that $N_i(v)\neq\emptyset$ for every $i\in[3q+1]$. 
	 
	 Now,  for each $j\in[q]$, pick a vertex $v_{3j}\in N_{3j}(v)$. Observe that $N(v_{3j})\subseteq N_{3j-1}(v)\cup N_{3j}(v) \cup N_{3j+1}(v)$, and hence, $$|N_{3j-1}(v)\cup N_{3j}(v) \cup N_{3j+1}(v)|\geq \delta(G)+1.$$ We also know that $|N_0(v)\cup N_1(v)|=|N(v)|+1\geq \delta(G)+1$. This proves the statement.
\end{proof}

The next result shows that we can make the exponent in Proposition~\ref{prop:bipartite} and Proposition~\ref{prop:non bipartite} only depend on the minimum degree of $G$. In order to prove the result we will  first apply a strategy similar to the one used in Propositions~\ref{prop:bipartite} and~\ref{prop:non bipartite}. If this strategy fails, we will have found a good structure in the host graph and then, forgetting about the earlier attempt at an embedding of $T$, we  make use of the structure  to embed the tree in a different way. 

\begin{proposition}\label{prop:connected-cte}
	For all $\alpha\in[\tfrac{1}{2},1)$,  $\varepsilon\in (0,10^{-8})$ and $M_0\in\mathbb{N}$, there exists $k_0\in\NN$ such that for all $n, k\ge k_0$ the following holds. \\ Let $G$ be a $n$-vertex graph with $\delta(G)\geq (1+100\sqrt{\varepsilon})\alpha k$ that has a connected $(\varepsilon,5\sqrt{\varepsilon})$-reduced graph $R$ with $|R|\le M_0$. If
	\begin{enumerate}[(a)]
		\item $R=(A,B)$ is bipartite and such that $|A|\ge (1+100\sqrt{\eps})k\cdot\tfrac{|R|}{n}$; or
		\item $R$ is non-bipartite and $n\ge (1+100\sqrt{\eps})k$;
		\end{enumerate}
	then $G$ contains every $k$-edge tree of maximum degree at most~$k^{\frac{1}{c}}$, where $c=18\lceil\tfrac{2}{\alpha}\rceil-5$. 
\end{proposition}

\begin{proof}
	Given $\alpha$, we define 
	\begin{equation*}
	\text{$d_1:=3\lceil\tfrac{2}{\alpha}\rceil-2$ and $d_2:=2(d_1+1)$,}
	\end{equation*}
	and observe that 
	$$c=3d_2+1.$$
	Given  $\eps$ and $M_0$,  let $k_0$ be the maximum of the outputs of Proposition~\ref{prop:bipartite} and Proposition~\ref{prop:non bipartite}, for input $\eps, d_2$ and $2M_0$.
	
	Let $G$ be as in Proposition~\ref{prop:connected-cte}.	Note that if $|V(G)|< (1+100\sqrt{\varepsilon})2k$, then Theorem~\ref{thm:Erdos} implies that $\diam(R)\le\lfloor\frac{6}{\alpha}\rfloor-1\le d_2$. Therefore (and because of Fact~\ref{fact:2}), we may apply either Proposition~\ref{prop:bipartite} or Proposition~\ref{prop:non bipartite}, together with Lemma~\ref{lem:matching2}, to conclude. Thus, from now on we will assume that
	\begin{equation}\label{bound_order}
	|V(G)|\ge(1+100\sqrt{\eps})2k.
	\end{equation} 
	
	Let $T$ be a tree on $k$ edges with $\Delta(T)\le k^{\frac{1}{c}}$ and root $T$ at any vertex. We partition $T$ using Proposition \ref{prop:cut3}, with $\beta:=\tfrac{\varepsilon n}{|R|}$, obtaining a set $S$ of seeds  and a family  $\mathcal{P}$ of pieces. We first try to emulate the embedding scheme used in the proof of Proposition \ref{prop:bipartite}.
	
	Consider the regular partition associated to the reduced graph $R$ of~$G$, and  divide each cluster $X$ into three sets $X_C,X_S,X_L$, with $|X_S|=|X_L|=\lceil10\sqrt{\varepsilon}|X|\rceil$. We are going to  embed $T$ in $|S|$ steps, letting $\phi$  denote the partial embedding defined so far. 
	
	At step $j$ we consider a vertex $s_j\in S$ not embedded yet, but whose parent $u_j$ is already embedded (except in the step $j=1$, in which case we embed the root of $T$ into any cluster of our choice). 
	We know that $\phi(u_j)$ is typical towards the $S$-slice of some adjacent cluster~$Q$. Embed $s_j$ in $Q_S$, choosing $\phi(s_j)$ typical to $U_L$ and to $U_S$, where $U$ is any neighbour of $Q$.
	
	 Now, suppose there is a good pair $(W,Z)$, that is, an edge $WZ$ such that both clusters $W$ and $Z$ have free space of size at least $5\sqrt{\varepsilon}|W|$, and additionally, $\text{dist}(U,W)\leq d_1$. Find a shortest path from $U$ to $W$, say $X_0 X_1\dots X_{t-1}X_t$, where $X_0=R$ and $X_t=W$ and, further, $t\le d_1$. 
	 
	 Consider a piece $P$ adjacent to $s_j$ that is not yet embedded. We map the root of $P$  into the neighbourhood of $\phi(s_j)$ in $(X_0)_L$. We then embed the first~$t$ levels of $P$ into the path $X_0 X_1\dots X_{t-1}X_t$, mapping the vertices from the $i^{\text{th}}$ level of $P$ into unoccupied vertices from $(X_{i})_L$ that are typical towards $(X_{i+1})_L$ and to $(X_{i+1})_S$, for each $i=0,\dots,t-1$ respectively. Finish the embedding of $P$, by mapping the remaining levels into the unoccupied vertices of $(W_C,Z_C)$. For this, we use Lemma~\ref{lem:T1}, mapping the vertices from the $t^{\text{th}}$ level of $P$ into~$W_C$ and picking all the images typical towards the $L$-slice and the $S$-slice of some adjacent cluster. We repeat this procedure for every not yet embedded piece adjacent to $s_j$ and then move on to the next seed.
	
	\smallskip
	
	If every step of this process is successful, then $T$ is satisfactorily embedded into $G$. However, it might  happen that the embedding cannot be completed, because at some step we could not find a good pair $(W,Z)$ at close distance. In that case, consider the seed $s^{*}$ where the process stopped and let $C^{*}$ be the cluster to which $s^{*}$ was assigned. Let us define $H$ as the subgraph of $R$ induced by all those clusters that lie at distance at most $d_1$ from~$C^*$. Further, let $\mathcal{S}$ be the set of all those clusters $C\in V(H)$ that have free space of size at least $5\sqrt{\varepsilon}|C|$. 
	
	 Note that, since the embedding could not be finished, then 
	 \begin{equation}\label{Sindep}
	 \text{$\mathcal{S}$ is an independent set. }
	 \end{equation}
	 
	 By applying Lemma~\ref{lem:dist}, with $q=\lceil\frac{2}{\alpha}\rceil-1$, and since $\delta(R)\ge (1+100\sqrt{\varepsilon})\alpha k\cdot \frac{|R|}{n}$ and by~\eqref{bound_order},  we deduce that $$|V(H)|>(1+100\sqrt{\varepsilon})2k\cdot \frac{|R|}{n}.$$ This is more than twice the space needed for embedding $T$. So, since we have embedded at most $k$ vertices before we declared the embedding to have failed, we conclude that
	 	 \begin{equation}\label{Sbig}
		 |V(\mathcal{S})|\geq (1+200\sqrt{\varepsilon})k\cdot \frac{|R|}{n}.
	 \end{equation}
	 
	Let us define $H'$ as the subgraph of $R$ induced by all  clusters at distance at most $d_1+1$ from $C^*$. So, $V(H')$ consists of~$V(H)$ together with the neighbours of $H$ in $R$.
 
 Forgetting about our previous attempt to embed $T$, we are now going to embed $T$ with the help of our earlier  propositions. We distinguish two cases, depending on whether  $H'$ is bipartite or not.

	\begin{flushleft}\textbf{Case 1:} $H'$ is nonbipartite. \end{flushleft}
	
	 Let $M$ be a matching in $H$ covering a maximal number of clusters from~$\mathcal S$. 
	 We claim that $|V(M)|\ge (1+100\sqrt{\varepsilon})k \cdot \frac{|R|}{n}$. Indeed, otherwise~\eqref{Sbig} implies that  there is a cluster $X\in V(\mathcal{S})\setminus{V(M)}$. By our choice of $M$, and because of~\eqref{Sindep}, we know that $x$ sees at most one endvertex of each edge from $M$, and no cluster outside $V(M)$. This contradicts the fact that $\deg_{H'}(X)\geq (1+100\sqrt{\varepsilon})\frac{k}{2}\cdot \frac{|R|}{n}$. 
	 
	 Hence, as $\diam(H')\leq 2(d_1+1) = d_2$, we can apply Proposition \ref{prop:non bipartite} to $H'$ and the subgraph of $G$ induced by the clusters of $H'$, and we are done.
	
	\begin{flushleft}\textbf{Case 2:} $H'$ is bipartite.\end{flushleft}
	
	Since $|V(H)|\geq(1+100\sqrt{\varepsilon})2k\cdot \frac{|R|}{n}$, one of the bipartition classes of $H'$, say~$A$, satisfies $|A\cap H|\geq (1+100\sqrt{\varepsilon})k\cdot \frac{|R|}{n}$. Since $\deg_{H'}(X)\geq (1+100\sqrt{\varepsilon})\frac{k}{2}\cdot \frac{|R|}{n}$ for each $X\in V(H)$, we can apply Proposition~\ref{prop:bipartite}, together with Remark \ref{rem:bipartite}, to obtain the embedding of $T$.
\end{proof}

\subsection{An approximate  Erd\H os--S\'os result:\\ The proof of Theorem~\ref{thm:ESap}}		\label{sec:ES}

As a quick application of the result from the previous subsection, we now prove our approximate version of the Erd\H os--S\'os conjecture for trees of bounded degree and dense host graphs.

\begin{proof}[Proof of Theorem~\ref{thm:ESap}]
We choose $\eps$ and $\eta$ such that  $0<\varepsilon\ll \eta\ll \delta$. Let $N_0,M_0$ be given by Lemma~\ref{reg:deg} for inputs $\eps$, $\eta$ and $m_0=\tfrac{1}{\eps}$. Set $n_0$ as the maximum between $N_0$ and the output of Proposition~\ref{prop:connected-cte}, with input $\varepsilon, \alpha:=\tfrac{1}{2}$ and $M_0$.

Now, let $G$ be a $n$-vertex graph with $d(G)\ge (1+\delta)k$. It is well-known that~$G$ has a subgraph $H$ with $d(H)\ge (1+\delta)k$ and minimum degree at least $(1+\delta)\frac{k}{2}$. By Lemma~\ref{reg:deg}, we can find a subgraph $H'\subseteq H$, which admits an $(\varepsilon,\eta)$-regular partition and satisfies $\delta(H')\ge (1+\tfrac{\delta}{2})\frac{k}{2}$ and $d(H')\ge (1+\tfrac{\delta}{2})k$. Let $R$ be the $(\varepsilon,\eta)$-reduced graph of $H'$. By averaging, there is a connected component $C$ of $R$ such that $d(C)\ge (1+\frac{\delta}{2})k\cdot \tfrac{|R|}{|H'|}$. In particular, $|\bigcup C|\ge (1+\frac{\delta}{2})k$. Clearly,  $C$  preserves the minimum degree in $H'$, that is,~$\delta(H'[C])\ge (1+\tfrac{\delta}{2})\tfrac{k}{2}$. 
Since $c=67=18\cdot 4-5$, we can apply Proposition~\ref{prop:connected-cte} to $C$ in order  to obtain the desired embedding of $T$.
\end{proof}

Clearly, this proof can easily be modified to give the following corollary (which will not be used anywhere in this paper). For $\eps>0$, we say that a subgraph $C$ of an $n$-vertex graph $G$ is {\em essentially $\eps n^2$-edge-connected} if, after deleting any set of at most $\eps n^2$ edges of $G$, there is a component containing~$C$.

\begin{corollary}\label{coro}
For every $\delta >0$ there is $n_0\in\NN$ such that for each $n$-vertex graph $G$  with $n\ge  n_0$ and  for each $k$ with  $\delta n\le k\le n$  the following holds. If~$G$ has an essentially $\delta n^2$-edge connected subgraph $C$ of size at least $(1+\delta )k$ with $\delta(C)>(1+\delta) \frac k2$, then $C$ contains every $k$-edge tree  of maximum degree at most $k^{\frac{1}{67}}$.	
\end{corollary}

\section{The key embedding lemma}\label{sec:general_lemma}

In the current section, we present and prove our key embedding lemma, namely Lemma~\ref{lem:superlemma}. This lemma describes a series of configurations which, if they appear in a graph $G$, allow us to embed any bounded degree tree of the right size into $G$. So, in a way the lemma can be seen as a compendium of favourable scenarios. 

Before stating the lemma we need two simple  definitions.

	\begin{definition}[$\theta$-see]
		Let $\theta\in(0,1)$. A vertex $x$ of a graph~$H$ $\theta$-sees a set $U\subseteq V(H)$ if it has at least $\theta |U|$ neighbours in~$U$.\\
		Furthermore,  if $ C$ is a component of some reduced graph of $H-x$, we say that $x$ $\theta$-sees $ C$ if $x$ has at least $\theta|\bigcup V(C)|$ neighbours in $\bigcup V(C)$.
	\end{definition}
	
	\begin{definition}[$(k,\theta)$-small, $(k,\theta)$-large]
		Let $k\in\NN$ and let $\theta\in(0,1)$. A nonbipartite  graph $G$ is said to be $(k,\theta)$-small if $|V(G)|<(1+\theta)k$. A bipartite graph $H=(A,B)$ is said to be $(k,\theta)$-small if $\max\{|A|,|B|\}<(1+\theta)k$.\\
		If a graph is not $(k,\theta)$-small, we will say that it is $(k,\theta)$-large.
	\end{definition}
	
	We are now ready for the key lemma.

	\begin{lemma}[Key embedding lemma]\label{lem:superlemma}
		For each $\alpha\in[\frac{1}{2},1)$,  for each $\eps \in(0,10^{-10})$ and for each $M_0\in\mathbb{N}$ there is $n_0\in\NN$ such that for all $n,k\ge n_0$ the following holds.\\
		Let $G$ be an $n$-vertex graph of minimum degree at least $(1+\sqrt[4]\eps)\alpha k$ and let $T$ be a $k$-edge tree whose maximum degree is bounded by $k^{\frac{1}{c}}$, where  $c=18\lceil\frac{2}{\alpha}\rceil-5$. Let $x \in V(G)$, and let $R$ be an $(\eps,5\sqrt\eps)$-reduced graph of $G-x$, with $|R|\le M_0$, such that at least one of the following conditions (I)--(IV) holds:
		\begin{enumerate}[(I)]
			\item\label{cond:1} $R$ has a $(k\cdot\tfrac{|R|}{n},\sqrt[4]\eps)$-large nonbipartite component; or
			\item\label{cond:2} $R$ has a $(k_1\cdot\tfrac{|R|}{n},\sqrt[4]\eps)$-large bipartite component, where $k_1$ is the size of the larger bipartition class of $T$; or
			\item\label{cond:3} $R$ has a $(\tfrac{2k}{3}\cdot \tfrac{|R|}{n},\sqrt[4]\eps)$-large bipartite component such that $x$ $\sqrt\eps$-sees both sides of the bipartition; or
			\item\label{cond:4} $x$ $\sqrt\eps$-sees two components $C_1$ and $C_2$ of $R$ in a way that one of the following holds:
			\begin{enumerate}[(a)]
				\item\label{cond:4a} $x$ sends at least one edge to a third component $C_3$ of $R$; 
				\item\label{cond:4b} there is $i\in\{1,2\}$ such that $C_i$ is nonbipartite and $(\tfrac{2k}{3}\cdot \tfrac{|R|}{n},\sqrt[4]\eps)$-large; 
				\item\label{cond:4c} there is $i\in\{1,2\}$ such that $C_i$ is bipartite and $x$ sees both sides of the bipartition; 
				\item\label{cond:4d} there is $i\in\{1,2\}$ such that $C_i$ is bipartite with parts $A$ and $B$, $\min\{|A|,|B|\}\ge(1+\sqrt[4]\eps)\tfrac{2k}{3}\cdot \tfrac{|R|}{n}$ and $x$ sees only one side of the bipartition;
				\item\label{cond:4e} $C_1$ and $C_2$ are bipartite with parts $A_1,B_1$ and $A_2,B_2$, respectively, $\min\{|A_1|,|B_2|\}\ge(1+\sqrt[4]\eps)\tfrac{2k}{3}\cdot \tfrac{|R|}{n}$ and $x$ does not see $B_1\cup B_2$.
			\end{enumerate}
		\end{enumerate}	
		Then $T$ embeds in $G$.
	\end{lemma}
	
	\begin{proof} Let $k'_0$ be the maximum of the outputs $k_0$ of Proposition~\ref{prop:connected-cte}, Corollary~\ref{emb:forest} and Corollary~\ref{emb:forest2}, for inputs  $\eps$, $d=\frac 6\alpha$ and $2M_0$, and choose $n_0:=k'_0+1$ as the numerical  output of Lemma~\ref{lem:superlemma}. 
		
		Now assume we are given  an $n$-vertex graph $G$  with $x\in V(G)$, and let $T$ be a $k$-edge tree as in Lemma~\ref{lem:superlemma}. Let $R$ be the $(\eps,5\sqrt\eps)$-reduced graph of $G-x$. An easy computation shows that 
		\begin{equation}\label{minim}
			\delta(R)\ge(1+\tfrac{1}{2}\sqrt[4]\eps)\alpha k\cdot \frac{|R|}{n}\ge(1+100\sqrt\eps)\alpha k\cdot \frac{|R|}{n},
		\end{equation}
		where the last inequality follows since $\eps\leq 10^{-10}$.
		Furthermore, note that $R$ must fulfill one of the conditions (I)--(IV) from Lemma~\ref{lem:superlemma}. 
		If $R$ contains a $(k\cdot \tfrac{|R|}{n},\sqrt[4]\eps)$-large nonbipartite component or a $(k_1\cdot\tfrac{|R|}{n},\sqrt[4]\eps)$-large bipartite component, then we can 
		conclude by Proposition~\ref{prop:connected-cte}. 
		
		So we can discard scenarios~\eqref{cond:1} and~\eqref{cond:2} from Lemma~\ref{lem:superlemma}. Therefore, by Theorem~\ref{thm:Erdos}, and by~\eqref{minim}, we can assume that every connected component $C$ of $R$ satisfies 
		\begin{equation}\label{diami}
			\diam(C) \leq \frac{3|V(C)|}{\delta (C)+1}\leq \frac{3(1+\sqrt[4]\eps)2k\cdot\tfrac{|R|}{n}}{(1+\tfrac{1}{2}\sqrt[4]\eps)\alpha k\cdot \tfrac{|R|}{n}}\leq \frac 6\alpha+1,
		\end{equation}
		and thus
		\begin{equation}\label{c_diami}
			c\ge 3\cdot \diam(C)+1.
		\end{equation}
		So, the maximum degree of $T$ and the diameter of the components are in the right relation to each other, meaning that  we could  apply Corollaries~\ref{emb:forest} and~\ref{emb:forest2} to each connected component of $R$ (if the other conditions of these corollaries hold).  
		
		\smallskip
		
		In order to embed $T$  under scenarios~\eqref{cond:3} and~\eqref{cond:4}, we use the results from Section~\ref{sec:cutT}. \\
		
		\textbf{Case 1 (scenario~\eqref{cond:3}):} $R$ has a $(\frac{2k}{3}\cdot \tfrac{|R|}{n},\sqrt[4]\eps)$-large bipartite component~$C$ such that $x$ $\sqrt\eps$-sees both sides of the bipartition.\\
		
		Applying Proposition~\ref{prop:cut5} to $T$, we obtain a cut-vertex $z_0\in V(T)$ and a proper $2$-colouring $c:V(T-z_0)\rightarrow\{0,1\}$ of $T-z_0$ such that
		\begin{equation*}
			\text{$|c_1|\leq |c_0|$, \ \ $|c_0|\leq\frac{2k}{3}$ \ \ and \ \ $|c_1|\leq \frac{k}{2}$.}
		\end{equation*}
		Let us note that, because of the bound on $k_0$, the number of components of $T-z_0$ is bounded by
		\begin{equation}\label{n:forest}\Delta(T)\le k^\frac{1}{c}\le \frac{\eps k}{M_0}\le \frac{\eps n}{|R|}.\end{equation}
		Now, we map $z_0$ into $x$. Recalling~\eqref{minim},~\eqref{diami},~\eqref{c_diami} and the fact that $T-z_0$ is a $(\frac{2k}{3},\frac{k}{2},\tfrac 1c)$-forest we can apply Corollary~\ref{emb:forest} to embed $T-z_0$ into $C$, and by~\eqref{n:forest} we may choose the images of the roots of $T-z_0$ as neighbours of $x$.\\
		
		\textbf{Case 2 (scenario~\eqref{cond:4}):} $x$ $\sqrt\eps$-sees two components $C_1$ and $C_2$ of $R$.\\
		
		Let $z_1\in V(T)$ be the vertex given by Lemma \ref{lem:cut1} applied to $T$, with any leaf $v$. Let $\mathscr{T}$ be the set of connected components of $T-z_1$. Then $\mathscr{T}$ is a family of at most $\Delta(T)$ rooted trees whose roots are neighbours of $z_1$ in $T$, and $|V(T')|\le \lceil\frac{k}{2}\rceil$ for every $T'\in\mathscr{T}$. 
		
		Apply Lemma~\ref{lem:num}~(i) to $\mathscr{T}$ to obtain a partition of $\mathscr{T}$ into three families of trees $\mathcal{F}_1,\mathcal{F}_2$ and $\mathcal{F}_3$, where $\mathcal{F}_3$ could be empty, such that 
		\begin{equation}\label{FFF}
			|V(\bigcup\mathcal{F}_3)|\le |V(\bigcup\mathcal{F}_2)|\le |V(\bigcup\mathcal{F}_1)|\le\left\lceil\frac{k}{2}\right\rceil.
		\end{equation}
		For later use, let us record here that 
		\begin{equation}\label{few_trees}
			\text{$|\mathcal{F}_1|+|\mathcal{F}_2|+|\mathcal{F}_3|\leq \Delta(T)\leq \frac{\eps n}{|R|}$.}
		\end{equation}	
		Furthermore, due to Remark~\ref{rem:F_3}, we know that
		\begin{equation}\label{oneinF3}
			|\mathcal{F}_3|\le 1.
		\end{equation}
		
		Similarly, applying Lemma~\ref{lem:num}~(ii) to $\mathscr{T}$ we obtain a partition of  $\mathscr T$ into two  families of trees $\mathscr{J}_1$ and $\mathscr{J}_2$ such that 
		\begin{equation}\label{JJJ}
			|V(\bigcup\mathscr{J}_2)|\le \frac k2 \text{\ \ and \ \ } |V(\bigcup\mathscr{J}_2)|\le|V(\bigcup\mathscr{J}_1)|\le\frac{2k}{3},
		\end{equation}
		and again, we know that
		\begin{equation}\label{few_treesJ}
			\text{$|\mathscr{J}_1|+|\mathscr{J}_2|\leq \Delta(T)\leq \frac{\eps n}{|R|}$.}
		\end{equation}	
		
		We split the remainder of the proof into six cases, according to which of the conditions \eqref{cond:4a}, \eqref{cond:4b}, \eqref{cond:4c}, \eqref{cond:4d} or \eqref{cond:4e} holds. Depending on the case we will make use of partition $\{\mathcal{F}_i\}_{i=1,2}$ or $\{\mathscr{J}_i\}_{i=1,2,3}$.\\ 

		\textbf{Case 2a (scenario~\eqref{cond:4a}):} $x$ $\sqrt\eps$-sees two components $C_1$, $C_2$ and sends at least one edge to a third component $C_3$.\\
		
		We embed $z_1$ into $x$, and then proceed to embed the roots of the trees from~$\mathcal{F}_i$  into the neighbourhood of $x$ in $C_i$, for each $i=1,2,3$. This is possible since by~\eqref{oneinF3}, there is at most one root to embed into $C_3$. Furthermore, by~\eqref{few_trees}, there are at most $\Delta (T)\leq\tfrac{\eps n}{|R|}$ roots to be embedded into $C_i$, for $i=1,2$.
		Finally, because of the minimum degree in $G$, and because of~\eqref{FFF}, we can greedily embed the remaining vertices of each forest  $\mathcal{F}_i$ into $C_i$. 
		\\
		
		\textbf{Case 2b (scenario~\eqref{cond:4b}):} $x$ $\sqrt\eps$-sees two components $C_1$ and $C_2$, and one of these components, say $C_1$, is nonbipartite and $(\frac{2k}{3}\cdot\tfrac{|R|}{n},\sqrt[4]\eps)$-large.\\
		
		We map $z_1$ into $x$, and then embed the roots of $\mathscr{J}_2$ into $C_2$ (we know that $x$ has enough neighbours in $C_2$ because of~\eqref{few_trees}).
		We then embed the rest of  $\bigcup\mathscr{J}_2$ greedily into $C_2$.
		
		For the trees from  $\mathscr{J}_1$, we can make use of Corollary~\ref{emb:forest2} and Lemma~\ref{lem:matching2}, whose conditions hold by~\eqref{minim},~\eqref{diami},~\eqref{c_diami} and~\eqref{JJJ}, to map $\bigcup\mathscr{J}_1$ to $C_1$. \\
		
		\textbf{Case 2c (scenario~\eqref{cond:4c}):} $x$ $\sqrt\eps$-sees two components $C_1$ and $C_2$, one of these components, say $C_1$, is bipartite, and $x$ sees both sides $A$, $B$ of the bipartition.\\
		
		First, we map $z_1$ into $x$ and then embed $\bigcup\mathcal{F}_1$ greedily into $C_2$ (embedding the roots into neighbours of $x$, as before). For the remaining forests, $\mathcal{F}_2$ and $\mathcal{F}_3$, observe that for any proper $2$-colouring of $\bigcup\mathcal{F}_2$ and $\bigcup\mathcal{F}_3$, and for any  $i=2,3$, the larger colour class of $\bigcup\mathcal{F}_i$ and the smaller colour class of $\bigcup\mathcal{F}_{5-i}$ add up to at most
		\begin{equation}\label{F1F2fit}
			|\bigcup\mathcal{F}_i| + \frac{|\bigcup\mathcal{F}_{5-i}|}{2} \leq \frac{|\bigcup\mathcal{F}_1|+|\bigcup\mathcal{F}_2|+|\bigcup\mathcal{F}_3|}{2} = \frac{k}{2}.
		\end{equation}			 
		
		Now, our aim is to embed the roots and all the even levels of $\bigcup\mathcal{F}_2$ into $A$, while embedding the odd levels into $B$. Moreover, we plan to embed  $\bigcup\mathcal{F}_3$ in a way that its  larger colour class goes to the same set as the smaller colour class of  $\bigcup\mathcal{F}_2$.
		
		As $x$ $\sqrt\eps$-sees $C_1$, we may assume that $x$ $\tfrac{\sqrt\eps}{2}$-sees $A$. 
		Moreover, since $x$ has at least one neighbour $b\in B$, and since $\bigcup\mathcal{F}_3$ has only one root  because of~\eqref{oneinF3}, we can choose whether we map the single root of $\bigcup\mathcal{F}_3$ into $b$, or into some neighbour of $x$ in $A$. We will make this choice according to our plan above (that is, it will depend on whether the even or the odd levels of $\bigcup\mathcal{F}_2$ contain more vertices).
		
		We then greedily embed the rest of $\bigcup\mathcal{F}_3$ into $C_1$. Now, we can make use of Corollary~\ref{emb:forest} together with Remark~\ref{rem:bipartite2F}, whose conditions hold by~\eqref{c_diami} and by~\eqref{F1F2fit}, to complete the embedding of $\bigcup\mathcal{F}_2$ into $C_1$, while avoiding the image of $\bigcup\mathcal{F}_3$.\\
		
		\textbf{Case 2d (scenario~\eqref{cond:4d}):}  $x$ $\sqrt\eps$-sees two components $C_1$ and $C_2$, one of them is bipartite with parts $A$ and $B$, $\min\{|A|,|B|\}\ge(1+\sqrt[4]\eps)\frac{2k}{3}\cdot\tfrac{|R|}{n}$ and $x$ sees only one side of the bipartition.\\
		
		Let us assume that $C_1$ is the bipartite component with parts $A$ and $B$ containing at least $(1+\sqrt[4]\eps)\frac{2k}{3}\cdot\tfrac{|R|}{n}$ clusters each and that $x$ only sees the set $A$. We map $z_1$ into $x$ and then embed $\bigcup\mathscr{J}_2$ greedily into $C_2$ (embedding the roots into neighbours of $x$, as before). Note that there are few roots of trees in $\mathscr{J}_1\cup\mathscr{J}_2$, because of~\eqref{few_treesJ}. 
		Since $\mathscr{J}_1$ is a $(\tfrac{2k}{3},\tfrac{k}{2},\tfrac{1}{c})$-forest, we may apply Corollary~\ref{emb:forest} so that we can embed $\bigcup\mathscr{J}_1$ into $C_1$ in a way that the images of its roots are neighbours of~$x$. This works because of~\eqref{JJJ}.\\
		
		\textbf{Case 2e (scenario~\eqref{cond:4e}):}  $x$ $\sqrt\eps$-sees two bipartite components $C_1$ and $C_2$, with parts $A_1,B_1$ and $A_2,B_2$ respectively, such that $\min\{|A_1|,|B_2|\}\ge(1+\sqrt[4]\eps)\frac{2k}{3}\cdot\tfrac{|R|}{n}$ and $x$ sees only $A_1$ and $A_2$.\\
		
		 We map $z_1$ into $x$, note  that $x$ $\sqrt\eps$-sees $A_1$ and $A_2$. Consider the colouring $\varphi$ that $T$ induces in $\bigcup\mathscr{J}_1$. If the roots of the trees in $\mathscr{J}_1$ are contained in the heavier colour class of $\varphi$, then we embed $\bigcup\mathscr{J}_1$ into $C_1$, otherwise we embed $\bigcup\mathscr{J}_1$ into $C_2$. In any case, and since $\mathscr{J}_1$ is a $(\tfrac{2k}{3},\tfrac k2,\tfrac1c)$-forest, we may use Corollary~\ref{emb:forest} to embed $\bigcup\mathscr{J}_1$ (taking care of mapping the roots into neighbours of $x$). Finally, we greedily embed $\bigcup\mathscr{J}_2$ into the remaining component.\\
		
		This completes the proof of Lemma~\ref{lem:superlemma}.
	\end{proof}

\section{Embedding  trees with  degree conditions}\label{s:mindegree} 

In this section we finally prove our main results, namely Theorems~\ref{main:1}, ~\ref{main:1-2} and~\ref{main:2}. All of them will be proved using~Lemma~\ref{lem:superlemma}, which, fortunately, makes all  these proofs quite straightforward. 

We begin by proving the principal result of this article,  Theorem~\ref{main:1}, in Section~\ref{principal}.
Then, we show Theorem~\ref{main:2} (the approximate version of $\frac 23$--conjecture) in Section~\ref{sec:2/3}. In Section~\ref{sec:Delta}, we show Theorem~\ref{main:1-2} (our extension of Theorem~\ref{main:1} to constant degree trees).

\subsection{An approximate version of Conjecture~\ref{2k,k/2}:\\ the proof of Theorem~\ref{main:1}}\label{principal}

Given $\delta\in(0,1)$, we set 
	$$ \varepsilon:=\frac{\delta^4}{10^{10}},\mbox{ and } \alpha:=\frac{1}{2}.$$
	Let $N_0,M_0$ be given by~Lemma~\ref{reg:deg}, with input $\varepsilon$ as defined above, $\eta:=5\sqrt\eps$ and $m_0:=\tfrac{1}{\eps}$, and let $n'_0$ be given by Lemma~\ref{lem:superlemma}, with input $\alpha, \varepsilon$ and $M_0$. We choose  $n_0:=(1-\eps)^{-1}\max\{n'_0,N_0\}+1$ as the numerical output of the theorem. 
	
Now, let $G$ be an $n$-vertex graph with minimum degree at least $(1+\delta)\frac{k}{2}$ and maximum degree at least $2(1+\delta)k$, where 
\begin{equation}\label{n_k_deltan}
n\ge k\ge \delta n
\end{equation}
and $n\geq n_0$. Let~$T$ be a $k$-edge tree with maximum degree at most $k^{\frac 1c}$, where $c=67=18\cdot 4-5$.

We apply Lemma~\ref{reg:deg} to $G-x$ so that we get a subgraph $G'\subseteq G-x$, with $|G'|\ge (1-\varepsilon)(n-1)$, that admits an $(\varepsilon,5\sqrt{\varepsilon})$-regular partition. Moreover, the minimum degree in $G'$ is at least
\begin{equation}\label{mindeg}\delta(G')\ge (1+\delta)\frac k2-(\eps+5\sqrt\eps)(n-1)-1\ge (1+\sqrt[4]\eps)\frac k2.\end{equation}
Let $R$ be the corresponding  $(\varepsilon,5\sqrt\eps)$-reduced graph of $G'$.
Our aim is to show that $R$ fulfills at least one of the conditions~\eqref{cond:1}--\eqref{cond:4} from Lemma~\ref{lem:superlemma}, for inputs $\alpha, \eps$ and $M_0$. We will assume that 
\begin{equation}\label{all_small}
\text{all the components of $R$ are $(k\cdot\tfrac{|R|}{|G'|},\sqrt[4]\eps)$-small,}
\end{equation}
as otherwise  either we have ~\eqref{cond:1} or~\eqref{cond:2} from Lemma~\ref{lem:superlemma}, and we are done. 

Since $G'$ misses less than $\varepsilon n$ vertices from $G$, we have that
\begin{equation}\label{maxim}
\deg_{G}(x,G')\ge 2(1+\tfrac{\delta}{2})k\ge 2(1+100\sqrt[4]\eps)k.
\end{equation}
 Suppose that~$x$ does not $\sqrt\eps$-see any component of $R$. By~\eqref{n_k_deltan} and by~\eqref{maxim}, this would mean that
\begin{equation}
2\delta n\le 2(1+\tfrac{\delta}{2})k\le\deg_{G}(x,G')=\sum_{C} \deg_{G}(x,\bigcup V(C))\le\sqrt\eps n,\tag{$\heartsuit$}
\end{equation}
a contradiction.
Therefore, there is some component $C_1$ of $R$ receiving more than $\sqrt\eps|\bigcup V(C_1)|$ edges from $x$. 

By~\eqref{all_small},
 $x$ can have at most $2(1+\sqrt[4]\eps)k$ neighbours in $C_1$. So by~\eqref{maxim}, there are more than $\sqrt[4]\eps k$ neighbours of $x$ outside $C_1$. Following the same reasoning as in $(\heartsuit)$, there must be a second component $C_2$ receiving at least $\sqrt\eps|\bigcup V(C_2)|$ edges from~$x$. We can assume that $x$ has no neighbours outside $C_1\cup C_2$, as otherwise condition~\eqref{cond:4a} from Lemma~\ref{lem:superlemma} holds. 

By~\eqref{maxim} and by symmetry, we can assume that
\begin{equation*}
\deg_{G}(x,\bigcup V( C_1)) \geq (1+\tfrac{\delta}{2})k.
\end{equation*}
In particular, we can again employ~\eqref{all_small} to see that $C_1$ is bipartite, and moreover $x$ has to see both classes of the bipartition.
Therefore, condition~\eqref{cond:4c} from Lemma~\ref{lem:superlemma} holds and the proof is finished. 

\subsection{An approximate version of the  $\frac{2}{3}$-conjecture:\\ the proof of Theorem~\ref{main:2}}\label{sec:2/3}

Given $\delta\in(0,1)$, we set 
	$$ \varepsilon:=\frac{\delta^4}{10^{10}},\mbox{ and } \alpha:=\frac{2}{3}.$$
	Let $N_0,M_0$ be given by~Lemma~\ref{reg:deg}, with input $\varepsilon$, and further input $\eta:=5\sqrt\eps$ and $m_0:=\tfrac{1}{\eps}$, and let $n'_0$ be given by Lemma~\ref{lem:superlemma}, with input $\alpha,\varepsilon$ and $M_0$. We choose  $n_0:=(1-\eps)^{-1}\max\{n'_0,N_0\}+1$ as the numerical output of the theorem. 
	
Let $G$ be an $n$-vertex graph as in Theorem~\ref{main:2}, with $n\ge n_0$, and let $T$ be a $k$-edge tree whose maximum degree is at most $k^{\frac 1c}$, where  $n\ge k\ge \delta n$ and $c=49=18\cdot 3-5$.\\

Let $x\in V(G)$ be a vertex of degree at least $(1+\delta)k$.  Since $n\ge n_0$, we can apply Lemma~\ref{reg:deg} to $G-x$ in order to get a subgraph $G'\subseteq G-x$, with $|G'|\ge (1-\varepsilon)(n-1)$, that admits an $(\varepsilon,5\sqrt{\varepsilon})$-regular partition. Let $R$ be the corresponding  $(\varepsilon,5\sqrt\eps)$-reduced graph of $G'$, we will assume that every component of $R$ is $(k\cdot\tfrac{|R|}{|G'|},\sqrt[4]\eps)$-small. An easy computation shows that
\begin{equation}\label{mindeg:2k3}\delta(G')\ge \left(1+\tfrac{\delta}{2}\right)\frac{2k}{3}\ge (1+100\sqrt[4]\eps)\frac{2k}{3},\end{equation}

because of the minimum degree in $G$. Also, note that $\deg_{G}(x, G')\ge(1+\frac{\delta}{2})k$. Following the same reasoning as in $(\heartsuit)$, and because of the degree of $x$, there is some component $C_1$ of $R$ such that $x$ $\sqrt\eps$-sees $C_1$. 
First, assume that $x$ has more than $(1+2\sqrt[4]\eps)k$ neighbours in $C_1$. Since $C_1$ is small, $C_1$ must be bipartite and $x$ must see at least a $\sqrt\eps$-portion of both sides of the bipartition, namely $A$ and $B$. 
Then, by \eqref{mindeg:2k3} we have $\max\{|A|,|B|\}\ge(1+\sqrt[4]\eps)\frac{2k}{3}\cdot\tfrac{|R|}{|G'|}$ and, therefore, $G'$ satisfies condition~\eqref{cond:3} from Lemma~\ref{lem:superlemma}.

\smallskip 

Now, we may assume that $x$ has less than $(1+2\sqrt[4]\eps)k$ neighbours in $C_1$. As in $(\heartsuit)$, we can calculate that there is a second component $C_2$ containing at least $\sqrt\eps|\bigcup V(C_2)|$ neighbours of $x$. We can assume that $x$ does not send edges to any other component, otherwise we are in case~\eqref{cond:4a} from Lemma~\ref{lem:superlemma}, and are done. 

 Also, by symmetry we can assume that $\deg_{G}(x,\bigcup V(C_1))\geq (1+\frac{\delta}{2})\frac{k}{2}$. Following the same reasoning as before we  conclude that  $|C_1|\ge(1+\frac{\delta}{2})\frac{2k}{3}\cdot \tfrac{|R|}{|G'|}$. In particular, if $C_1$ is nonbipartite, then $G'$ satisfies condition~\eqref{cond:4b} from Lemma~\ref{lem:superlemma} and we are done.
 
  So we may suppose that $C_1$ is bipartite.  If $x$  sees both sides of the bipartition,  condition~\eqref{cond:4c} from Lemma~\ref{lem:superlemma} holds, so let us assume this is not the case. The minimum degree tells us that one of the sides of the bipartition of $C_1$ has size at least $(1+\frac{\delta}{2})\frac{2k}{3}\cdot \tfrac{|R|}{|G'|}$ clusters, and we can argue similarly for the other side of the bipartition. This means that $G'$ satisfies condition~\eqref{cond:4d} from Lemma~\ref{lem:superlemma}, which completes the proof.

\subsection{Embedding constant degree trees:\\ the proof of Theorem~\ref{main:1-2}}\label{sec:Delta}
	
Given $\delta\in(0,1)$ and $\Delta\ge 2$, we set 
	$$ \varepsilon:=\frac{\delta^4{}}{ 10^{10}}\mbox{ and }\alpha:=\frac{1}{2}.$$
	Let $N_0,M_0$ be given by~Lemma~\ref{reg:deg}, with input $\varepsilon,\eta:=5\sqrt\eps$ and $m_0:=\tfrac{1}{\eps}$, and let $n'_0$ be given by Lemma~\ref{lem:superlemma}, with input $\alpha,\varepsilon$ and $M_0$. We choose  $n_0:=(1-\eps)^{-1}\max\{n'_0,N_0\}+1$ as the numerical output of the theorem. 
	
	Let $G$ be an $n$-vertex graph as in Theorem~\ref{main:1-2}, where $n\geq n_0$, and let $T$ be a $k$-edge tree whose maximum degree is bounded by~$\Delta$, where $k$ is such that $n\ge k\ge \delta n$. Let $x\in V(G)$ be a vertex of degree at least $2(\frac{\Delta-1}{\Delta}+\delta)k$. We apply Lemma~\ref{reg:deg} to $G-x$ and we obtain a subgraph $G'$, with $|G'|\ge (1-\varepsilon)(n-1)$, that admits an $(\varepsilon,5\sqrt{\eps})$-regular partition. Let $R$ be the corresponding $(\eps,5\sqrt\eps)$-reduced graph.

Observe that each $k$-edge tree $T$ with maximum degree at most $\Delta$ will satisfy
\begin{equation}\label{bip_bound}
k_1\leq \frac{\Delta-1}{\Delta}k,
\end{equation}
where $k_1$ is the size of the larger bipartition class of $T$. We can discard scenarios~\eqref{cond:1} and~\eqref{cond:2} and therefore assume that
\begin{equation}\label{all_nbip_small}
\text{all nonbipartite components of $R$ are $(k\cdot \tfrac{|R|}{|G'|},\sqrt[4]\eps)$-small,}
\end{equation} 
and, by~\eqref{bip_bound},
\begin{equation}\label{all_bip_small}
\text{all bipartite components of $R$ are $(\tfrac{\Delta-1}{\Delta}k\cdot\tfrac{|R|}{|G'|},\sqrt[4]\eps)$-small.}
\end{equation} 
As we removed only few vertices from $G$, it is clear that $x$ has at least $2(\tfrac{\Delta-1}{\Delta}+\tfrac{\delta}{2})k$ neighbours in $G'$. This, together with~\eqref{all_nbip_small} and~\eqref{all_bip_small}, implies that there are components $C_1$ and $C_2$ of $R$ such that
\begin{equation*}
\text{$\deg_{G}(x,\bigcup V(C_i))\ge \sqrt\eps|\bigcup V(C_i)|$,\ \ for $i=1,2$.}
\end{equation*}
Moreover, we may assume that $x$ does not see any other components, otherwise $G'$ satisfies condition~\eqref{cond:4a} from Lemma~\ref{lem:superlemma} and we are done. First, suppose that $\Delta=2$, that is, $T$ is a path of length $k$. In this case, we  choose a $\tfrac{k}{2}$-separator $z$ of $T$ and embed $z$ into $x$. After that, we can greedily embed each component of $T-z$ into $C_1$ and $C_2$, respectively. 

Now, suppose that $\Delta\ge 3$. By symmetry, we may assume that
\begin{equation}\label{deg_c1}
\deg_{G}(x,\bigcup V(C_1))\geq (\tfrac{\Delta-1}{\Delta}+\tfrac{\delta}{2})k.
\end{equation}
If $C_1$ is nonbipartite, $G'$ satisfies condition~\eqref{cond:4b} from Lemma~\ref{lem:superlemma} as $\Delta\ge 3$. If $C_1$ is bipartite with parts $A$ and $B$, we can employ~\eqref{all_bip_small} together with~\eqref{deg_c1} to conclude that $G'$ satisfies condition~\eqref{cond:4c} from Lemma~\ref{lem:superlemma}. This  concludes the proof.

\section{Conclusion}\label{sec:conj}

\subsection{Constant degree trees}\label{sec:const}

For  trees whose maximum degree is bounded by an absolute constant $\Delta$, we believe that  the bound on the maximum degree of the host graph from Conjecture~\ref{2k,k/2} can be weakened.

\begin{conjecture}\label{conj:const}
Let $k, \Delta\in\mathbb N$, and let $G$ be a graph with $\delta(G)\geq \frac k2$ and $\Delta (G)\geq 2(1-\frac 1\Delta)k$. Then $G$ contains all trees $T$ with $k$ edges with $\Delta (T)\leq \Delta$. 
\end{conjecture}

Evidence for this conjecture is given by Theorem~\ref{main:1-2}.
Note that the bounds on the maximum and minimum degree of $G$ are close to best possible, which can be seen by considering the following example.

\begin{example}\label{exampleDelta}
Let $\Delta\ge 2$ and let $T$ be a tree on an odd number of levels, where the root has degree  $\Delta +1$, also every vertex in an odd level  has degree  $\Delta +1$, and every vertex in an even level, except for the leaves and except for the root, has degree 2. \\
Setting $k:=|T|-1$, we can calculate that there are 
$\frac{\Delta}{\Delta +1}k+1$ vertices in even levels, that is, there are $\frac{\Delta}{\Delta +1}k$ vertices in positive  even levels. (For this, note that we can cover all but one vertex of $T$ by disjoint stars $K_{1,\Delta}$ centered at the vertices from  odd levels.) \\
Consider the complete bipartite graph $H$ with bipartition classes  $A$, $B$ of sizes $|A|=(\frac{\Delta}{\Delta +1})^2k$ and $|B|=\frac{\Delta}{2(\Delta +1)}k$. Take two copies of $H$, and join a new vertex $x$ to each vertex in the larger bipartition classes. The obtained graph~$G$ misses the degree conditions from Conjecture~\ref{conj:const} only by a factor of $\tfrac{ \Delta}{\Delta+1}$. \\
Observe that
$G$ does not contain $T$, because either one of the sets $B$ would have to receive at least half of all vertices in positive even levels of $T$, or one of the sets $A$ would have to receive the root of $T$, plus at least a  $\frac{\Delta}{\Delta +1}$-fraction of all vertices in positive even levels of $T$.
\end{example}

It is easy to see that the tree from Example~\ref{exampleDelta} is the most unbalanced  tree whose maximum degree is bounded by the constant $\Delta$.

\subsection{A variation of the $2k$--$\frac k2$ conjecture}\label{ell}

We believe that a more general statement than the one given in Conjecture~\ref{2k,k/2} should be true. More precisely, if we relax our bound on the maximum degree of the host graph, while at the same time asking for a stronger bound on the minimum degree of the host graph, the conclusion of Conjecture~\ref{2k,k/2} should still hold true. 
Quantitatively speaking, we conjecture the following (see also~\cite{BPS3}).

\begin{conjecture}\label{conj:ell}
Let $k\in\mathbb N$, let $0\le\alpha\le \frac 12$ and let $G$ be a graph with $\delta(G)\geq  (1+\alpha)\frac k2$ and $\Delta (G)\geq 2(1-\alpha)k$. Then $G$ contains all trees with $k$ edges. 
\end{conjecture}

Note that for $\alpha= 0$,  the bounds  from Conjecture~\ref{conj:ell} coincide with the bounds from Conjecture~\ref{2k,k/2}. 
For all $\alpha\in [\frac 13, \frac 12]$,  Conjecture~\ref{conj:ell}  follows from the $\frac 23$-conjecture. 
Also,  it is easy to see the conjecture  holds for stars, double-stars and paths (see~\cite{BPS3}).

We will discuss Conjecture~\ref{conj:ell} in more detail in the forthcoming manuscript~\cite{BPS3}. In particular,  we show that the conjecture is asymptotically true for bounded degree trees in dense host graphs, and we show that Conjecture~\ref{conj:ell} is asymptotically best possible in the following sense.

\begin{proposition}\label{prop:ell'}$\!\!${\rm\bf\cite{BPS3}}
For each odd $\ell\in\mathbb N$ with $\ell\ge 3$, and for each $\gamma>0$ there are $k\in \mathbb N$, a $k$-edge tree~$T$, and a graph $G$ with $\delta(G)\geq  (1+\frac 1\ell -\gamma)\frac k2$ and $\Delta (G)\geq 2(1-\frac 1\ell-\gamma)k$ such that $T$ is not a subgraph of $ G$. 
\end{proposition}

In particular, this disproves a conjecture from~\cite{rohzon}.

\subsection{Maximum degree $\frac{4k}{3}$}	\label{4/3}

Host graphs of maximum degree close to $\frac 43k$ have already appeared in the examples given in~\cite{2k3:2016} for the sharpness of the $\frac 23$-conjecture. In~\cite{2k3:2016}, two examples are given: a balanced complete bipartite graph of order $\frac 43k-4$ to which a universal vertex is added, and the union of two disjoint copies of $K_{\frac 23k-1}$, to which a universal vertex is added. These graphs have maximum degree $\frac 43k-4$ and $\frac 43k-2$, respectively, and minimum degree $\frac 23 k-1$. The tree consisting of three stars of order~$\frac k3$, whose centres are joined to a new vertex $v$, does not fit into either of the graphs described above. Let us note that also the tree obtained from taking three paths of order~$\frac k3$ and joining their first vertices to a new central vertex does not fit into the second host graph described above.

It follows from Proposition~\ref{prop:ell'} that requiring a maximum degree of at least $ck$, for any $c<2$ (in particular for $c=\frac{4}{3}$), and a minimum degree of at least  $\frac k2$ is not enough to guarantee all trees with $k$ edges as subgraphs. Nevertheless, graphs that look very much like the graph  from Example~\ref{example1} (or the similar example underlying Proposition~\ref{prop:ell'}) might be the only obstructions to embed a tree with~$k$ edges into a graph of maximum degree greater than $\frac{4k}{3}$ and minimum degree greater than~$\frac{k}{2}$. 
In the forthcoming~\cite{BPS3}, this suspicion is partially confirmed. We show that asymptotically, any bounded degree tree embeds in any host graph of maximum degree greater than $\frac{4k}{3}$ and minimum degree greater than $\frac{k}{2}$, as long as it does not too closely resemble the graph from Example~\ref{example1} or a similar example.

\subsection{Large first and  second neighbourhood}\label{2nd}

A different way of avoiding Example \ref{example1} would be to impose a condition on the size of the second neighbourhood of a large degree vertex. We explore this approach in the present section. 

Let $N_2(x)$ denote the second neighbourhood of vertex $x\in V(G)$.

\begin{conjecture}
Let $k\in\mathbb{N}$, and let $G$ be a graph with $\delta(G)\geq \frac k2$, which has a vertex $x$ such that  $\min\{|N(x)|, |N_2 (x)|\}\geq \frac{4k}3$. Then $G$ contains all trees with $k$ edges. 
\end{conjecture}

We can prove an approximate version of this for bounded degree trees, and dense host graphs.

\begin{theorem}\label{main:3} Let $\delta \in(0,1)$. There exist $k_0\in \NN$ such that for all $n,k\ge k_0$ with $n\ge k\ge \delta n$ the following holds for every $n$-vertex graph $G$ with $\delta(G)\geq(1+\delta)\frac{k}{2}$ and every  $k$-edge tree $T$ with $\Delta(T)\le k^{\frac{1}{67}}$.\\ If there is a vertex $x\in V(G)$ with $\min\{|N(x)|, |N_2(x)|\}\geq(1+\delta)\frac{4k}{3}$, then $T$ embeds in $G$.
\end{theorem}

\begin{proof}
Given $\delta\in(0,1)$, set $\eps:=\frac{\delta^4}{10^{10}} \  \mbox{ and } \ \alpha:=\frac{1}{2}.$
	Let $N_0,M_0$ be given by~Lemma~\ref{reg:deg}, with input $\varepsilon$, $\eta:=5\sqrt{\eps}$ and $m_0:=\frac{1}{\eps}$. Let $n'_0$ be given by Lemma~\ref{lem:superlemma}, with input $\alpha$, $\varepsilon$ and $M_0$. Choose $n_0:=(1-\eps)^{-1}\max\{n'_0,N_0\}+1$ as the numerical output of the theorem.
	
	Now, let $G$, $x\in V(G)$ and $T$ be given as in Theorem~\ref{main:3}.
	We apply Lemma~\ref{reg:deg} to $G-x$, obtaining subgraph $G'$ which admits an $(\varepsilon,5\sqrt{\varepsilon})$-regular partition, with corresponding  $(\varepsilon,5\sqrt\eps)$-reduced graph $R$. Note that  
	$$\delta(R)\geq (1+\frac{\delta}2)\frac{k}{2}\cdot\frac{|R|}{|G'|}\ge(1+100\sqrt[4]\eps)\frac{k}{2}\cdot\frac{|R|}{|G'|}.$$
 
We quickly deduce that $x$  $\sqrt\eps$-sees two components $C_1$ and $C_2$ of $R$, and does not see any other component, as otherwise we are in  one of scenarios~\eqref{cond:1}, ~\eqref{cond:2}, ~\eqref{cond:3}, or~\eqref{cond:4a} from Lemma~\ref{lem:superlemma}, and can thus  embed $T$. In particular, we have $N_2(x)\subseteq \bigcup V(C_1\cup C_2)$.

Symmetry allows us to assume that $\deg_{G}(x,\bigcup V(C_1))\ge \deg_{G}(x,\bigcup V(C_2))$. In particular, $\deg_{G}(x,\bigcup V(C_1))\ge (1+\frac{\delta}{2})\frac{2k}{3}$. If we are in neither of scenarios~\eqref{cond:4b} or~\eqref{cond:4c} from Lemma~\ref{lem:superlemma}, then
 $C_1$ is bipartite, say with parts $A_1$ and $B_1$, and  $x$ sees only one side of $C_1$, say $A_1$. 
 
Now, if $|B_1|\ge (1+\sqrt[4]\eps)\frac{2k}{3}\cdot\frac{|R|}{|G'|}$, we are in scenario~\eqref{cond:4d} from Lemma~\ref{lem:superlemma}, and can  embed $T$. 
So assume otherwise. Then,  $C_2$ contains at least $(1+\sqrt[4]\eps)\frac{2k}{3}$ vertices from  $N_2(x)$.
If we are in neither of scenarios~\eqref{cond:4b} or~\eqref{cond:4c} from Lemma~\ref{lem:superlemma}, then
 $C_2$ is bipartite, say with parts $A_2$ and $B_2$. Furthermore, $x$ does not see $B_2$ and $|B_2|\ge (1+\sqrt[4]\eps)\frac{2k}{3}\cdot\tfrac{|R|}{|G'|}$.
If $|A_2|\ge (1+\sqrt[4]\eps)\frac{2k}{3}\cdot\tfrac{|R|}{|G'|}$, we are in scenario~\eqref{cond:4d} from Lemma~\ref{lem:superlemma}. Otherwise, we are scenario~\eqref{cond:4e} from Lemma~\ref{lem:superlemma}, which concludes our proof.	 	
\end{proof}
\medskip

\subsection{Supersaturation}\label{supersat}
Given a graph $F$, the function $\ex(n,F)$ is defined as the maximum number an $n$-vertex graph can have without containing a copy of $F$.
 The supersaturation phenomenon is that once a graph on $n$ vertices has substantially more than $\ex(n,F)$ edges, a large number of copies of $F$ appear. In our context,  the Erd\H os--S\'os conjecture would imply that 
$$\ex(n,{T}_k)= \tfrac{1}{2}(k-1)n,$$
 where ${T}_k$ is any fixed $k$-edge tree. In view of our Theorem~\ref{thm:ESap}, the only question we can answer here is how many copies of $T_k$ the theorem guarantees, where $T$ is of linear size and bounded degree. We can see from Lemma~\ref{lem:T1}, that if we are embedding $t$ vertices into an $(\varepsilon,5\sqrt{\varepsilon})$-regular pair, then the number of ways one can embed those vertices is at least
 
 \begin{equation*}\label{sup:regpair}2^{t}\varepsilon^{t}\left(\frac{n}{\ell}\right)^t\ge \left(\frac{2\varepsilon}{M_0}\right)^{t}n^{t},\end{equation*}
 where $\ell$ and $M_0$ are as in Lemma~\ref{reg:deg}.
So, given $\delta>0$, given a graph $G$ with $e(G)\ge (1+\delta)\tfrac{k}{2}n$ and $n\ge n_0$, and given any $k$-edge tree $T$ with bounded maximum degree,  one can easily deduce from the proof of Theorem~\ref{thm:ESap}
 that~$G$ contains at least
 $c^{k+1}n^{k+1}$
 copies of $T$, where $c$ is a very small constant that only depends on $\delta$. 
 
 Similarly, given $\delta>0$ and an $n$-vertex graph $G$ as in Theorem~\ref{main:1}, one can deduce that the number of copies of any $k$-edge tree $T$ with bounded maximum degree is at least $C^kn^k,$ where $C$ is a very small constant depending on $\delta$. (The exponent drops to $k$ because we  map a cutvertex $z$ of $T$ into some maximum degree vertex of $G$. So, if the number of vertices of degree at least $(1+\delta)2k$ in $G$ is not linear in $n$,  we cannot expect as many copies of $T$ as in the situation of Theorem~\ref{thm:ESap}.)

\subsection{Extremal graphs}\label{resil}

What happens to our results if the minimum/average/maximum degree of a graph is slightly below the thresholds given in the theorems? Can we still embed any $k$-edge tree $T$ of bounded maximum degree, and if not, can we describe the structure of $G$?

It is possible to deduce\footnote{See the forthcoming~\cite{BPS2} for a more precise argument.}
 from the proofs of Propositions~\ref{prop:bipartite}, \ref{prop:non bipartite} and \ref{prop:connected-cte} that if the minimum degree of $G$ is close to the bound $\frac k2$, then we are still able to embed the tree $T$, unless our graph $G$ has the following structure: each component of $G$ has size less than $k$, or every large component of $G$ is the disjoint union of sets $A$ and $B$ such that $A$ is almost independent and almost all edges are present between $A$ and $B$. In the latter case, we know that furthermore, either $B$ is almost independent, and both $A$ and $B$ are smaller than $k$, which means we are not able to embed  trees that are very unbalanced, or  $B$ has size less than $\tfrac k2$, in which case we cannot embed  trees that are almost perfectly balanced.  
 
\smallskip

\paragraph{Extremal graphs for Theorem~\ref{thm:ESap}.}
For Theorem~\ref{thm:ESap} (our approximate version of the Erd\H os--S\'os conjecture), one can prove that the components of~$G$ are either almost complete and of size roughly $k$, or bipartite and as  described above. 

The two situations correspond to the classical examples for the sharpness of the Erd\H os--S\'os conjecture: the union of~$\frac nk$ complete graphs on~$k$ vertices (if~$k$ does not divide $n$, add in a smaller complete graph); a complete bipartite graph with bipartition classes of size almost $k$; and a complete graph on~$n$ vertices from which all edges inside a set~$A$ have been deleted, where $|A|=n-\lfloor\frac k2\rfloor +1$ (note that the second example has slightly less edges).

\smallskip

\paragraph{Extremal graphs for Theorems~\ref{main:1} and~\ref{main:1-2}.}
For Theorem~\ref{main:1} (our approximate version of Conjecture~\ref{2k,k/2}), we arrive at the following situation if we cannot embed $T$ as planned. Any maximum degree vertex $x$ of $G$  sees exactly two components $C_1, C_2$ of $G'-x$. Both $C_1, C_2$ are bipartite and $x$ sees only one side of each  $C_i$. This is precisely the situation from Example~\ref{example1} (although $G$ might have more components that are simply too small).

For Theorem~\ref{main:1-2} (the variant of Theorem~\ref{main:1} for constant degree trees), we arrive at an example similar to the one  from Example~\ref{example1}, with the size of the sets slightly adjusted. 

\smallskip

\paragraph{Extremal graphs for Theorem~\ref{main:2}.}
For Theorem~\ref{main:2} (our approximate version of the $\frac 23$-conjecture), it might again happen that $G$ consists of components that are each to small to accommodate~$T$. 

It might also happen that any maximum degree vertex $x$ sees both sides $A$ and~$B$ of a bipartite component $C$ of $G'-x$, and both $A$ and $B$ are smaller than~$\frac 23k$. This is exactly the situation described in one of the examples from~\cite{2k3:2016} for the sharpness of  the $\frac 23$-conjecture (a complete bipartite graph with sides of size $\frac 23k -2$ plus a universal vertex).

If that is not the case, then any maximum degree vertex $x$ sees at least two components $C_1, C_2$ of $G'-x$, and if it sees a third component, the embedding succeeds. It might now happen that $x$ sees two almost complete components of size almost $\frac 23k$, and we cannot embed $T$. This corresponds to the second example from~\cite{2k3:2016} for the sharpness of  the $\frac 23$-conjecture (two complete  graphs  of order $\frac 23k -1$ plus a universal vertex).

However, it might also happen that $C_1$ is bipartite, with bipartition classes $A_1$ and $B_1$, such that one of these classes is slightly smaller than $\frac 23k$, and such that $x$ sees only one side of the bipartition, say $A_1$. In this case, it might again be impossible to embed $T$. 

The latter situation is described in the following example. Note that this example is  quite similar to Example~\ref{example1}, but  allows for two different shapes of the second component. 

\begin{example}\label{example:2/3}
Assume that $k$ is divisible by $3$. For $i=1,2$, let $C_i=(A_i, B_i)$ be a complete bipartite graph with $|A_i|=\frac 23k-2$ and $|B_i|=\frac 23k-3$. Let $C'_2$ be a complete graph with $|C'_2|=\frac 23k-1$. \\
Let $G$ be obtained from $C_1\cup C_2$ by adding a new vertex $x$ adjacent to all of $A_1\cup A_2$. Let $G'$ be obtained from $C_1\cup C_2'$ by adding a new vertex $x$ adjacent to all of $A_1\cup V(C'_2)$. \\
Then the tree $T$ obtained from taking three stars of order $\frac k3$ and adding a new vertex adjacent to their centres, does not embed in either of $G$, $G'$.
\end{example}

Note that if we take the sizes of the sets $A_i$ in Example~\ref{example:2/3} to be $\frac 23k-3$ instead of $\frac 23k-2$, then also the tree
 $T'$, obtained from taking three stars of order $\frac k3$ and adding a new vertex adjacent to one leaf from each of the stars,  does not embed in either of $G$, $G'$.

The reader might object that Example~\ref{example:2/3} does not really apply to the situation from Theorem~\ref{main:2}, as the tree $T$ from the example has maximum degree $\frac k3$. However, we can modify the tree a  bit in order to comply with the maximum degree requirement, keeping it as unbalanced as possible, and maintaining the property that it consists of three subtrees of the same order plus a vertex that joins them. Then Example~\ref{example:2/3} still works, after adjusting the sizes of the sets $A_i$ and $B_i$.

\bibliographystyle{acm}
\bibliography{trees}

\end{document}